\documentclass[12pt]{amsart}

\usepackage{amsmath,amsxtra,amssymb,latexsym,
epsfig,amscd,amsthm,subfigure,fancybox,float,epsfig}
\usepackage[mathscr]{eucal}
\usepackage{graphicx}
\usepackage{epsfig} 
\usepackage{epstopdf}
\usepackage{cases}
\usepackage{multicol,color}
\def\para#1{\vskip .4\baselineskip\noindent{\bf #1}}
\usepackage{placeins}

\usepackage{hyperref}
\usepackage{epstopdf}
\usepackage{cases}
\usepackage{color}
\usepackage{hyperref}
\hypersetup{
    colorlinks=true,
    linkcolor=blue,
    filecolor=magenta,
    urlcolor=cyan,
    citecolor=blue,
}

\newtheorem{thm}{Theorem}[section]
\newtheorem {asp}{Assumption}[section]
\newtheorem{lm}{Lemma}[section]

\newtheorem{pron}{Proposition}[section]
\newtheorem{deff}{Definition}[section]

\newtheorem{lem}{Lemma}[section]
\newtheorem{prop}{Proposition}[section]

\theoremstyle{remark}
\newtheorem{rem}{Remark}
\newtheorem{example}{Example}[section]
\numberwithin{equation}{section}

\newcommand{\eps}{\varepsilon}

\newcommand{\M}{\mathcal{M}}
\newcommand{\F}{\mathcal{F}}

\newcommand{\E}{\mathbb{E}}

\newcommand{\PP}{\mathbb{P}}

\newcommand{\R}{\mathbb{R}}

\newcommand{\Lom}{\mathcal{L}}

\newcommand{\abs}[1]{\left\vert#1\right\vert}

\numberwithin{equation}{section}

\newcommand{\wdt}{\widetilde}
\newcommand{\PPi}{\PP_{u,v,\pi}}
\newcommand{\Ei}{\E_{u,v,\pi}}

\newcommand{\bed}{\begin{displaymath}}
\newcommand{\eed}{\end{displaymath}}
\newcommand{\bea}{\bed\begin{array}{rl}}
\newcommand{\eea}{\end{array}\eed}

\newcommand{\barray}{\begin{array}{ll}}
\newcommand{\earray}{\end{array}}

\def\disp{\displaystyle}

\newcommand{\1}{\boldsymbol{1}}

\newcommand{\bmu}{\boldsymbol{\mu}}
\newcommand{\bdelta}{\boldsymbol{\delta}}

\def\bar{\overline}
\def\hat{\widehat}
\def\a.s{\text{\;a.s.\;}}

\title{
Stochastic Epidemic SIR Models with Hidden States}

\author[N. Du]{Nguyen Du }
\address{Department of Mathematics-Mechanics-Informatics \\
Vietnam National University of Science\\
34 Nguyen Trai\\
 Thanh Xuan, Ha Noi\\
Vietnam}
\email{nhdu@viasm.edu.vn}

\author[A. Hening]{Alexandru Hening }
\address{Department of Mathematics\\
Texas A\&M University\\
Mailstop 3368\\
College Station, TX 77843-3368\\
United States
}
\email{ahening@tamu.edu}

\author[N. Nguyen]{ Nhu Nguyen }
\address{Department of Mathematics \\
University of Connecticut\\
341 Mansfield Road U1009\\
Storrs, Connecticut 06269-1009\\
United States. {\rm The research of this author was supported in part by the National Science Foundation.}}
\email{nguyen.nhu@uconn.edu }

\author[G. Yin]{George Yin}
\address{Department of Mathematics \\
University of Connecticut\\
341 Mansfield Road U1009\\
Storrs, Connecticut 06269-1009\\
United States. {\rm The research of this author was supported in part by the National Science Foundation.}}
\email{gyin@uconn.edu}

\keywords {SIR model; Extinction;  Permanence; Stationary Distribution; Ergodicity}
\subjclass[2010]{34C12, 60H10, 92D25}

\begin{document}
\begin{abstract}
This paper focuses on and
analyzes  realistic SIR models that take  stochasticity
into account. The proposed systems
are applicable to
most incidence rates that are used in the literature including the bilinear incidence rate, the Beddington-DeAngelis incidence rate, and a Holling type II functional response.
 Given that many diseases can lead to asymptomatic infections, we look at a system of stochastic differential equations that also includes a class of hidden state individuals, for which the infection status is unknown. We assume that the direct observation of the percentage of hidden state individuals that are infected, $\alpha(t)$, is not given and only
 a noise-corrupted observation
 process is available.
 Using
 the nonlinear filtering techniques
 in conjunction with an invasion type analysis (or analysis using Lyapunov exponents from the dynamical system point of view), this paper
 proves that the long-term behavior of the disease is governed by a threshold $\lambda\in \R$
  that depends on the model parameters. It turns out that if $\lambda<0$ the number $I(t)$ of infected individuals converges to zero exponentially fast, or the extinction happens. In contrast,
  if $\lambda>0$, the infection is endemic and the system 
  is permanent.
  We showcase our results by applying them in specific illuminating examples.
Numerical simulations are also given to illustrate our results.
\end{abstract}

\maketitle

\section{Introduction}\label{sec:int}

The SIR epidemic models introduced first by \cite{KER1,KER2} look at the dynamics of susceptible, infected, and recovered individuals, whose densities at the time $t$ are denoted by $S(t)$, $I(t)$, and $R(t)$, respectively. In the absence of random effects,
the dynamics  are
described
by the following system of differential equations
\begin{equation}\label{1.0}
\begin{cases}
dS(t)=\big[a_1-\mu_SS(t)-F(S(t),I(t))\big]dt, \\
dI(t)=\big[-(\mu_I+r)I(t) + F(S(t),I(t)\big]dt,\\
dR(t)=\big[-\mu_RR(t)+rI(t)\big]dt.
\end{cases}
\end{equation}
Here $a_1>0$ is the recruitment rate of the population, $\mu_S,\mu_I,\mu_R>0$ are the death rates of the susceptible, infected, and recovered individuals, $r>0$ is the recovery rate of the infected individuals and $F(S(t),I(t))$ is the incidence rate.
The dynamics of recovered individuals have no effect on that of the disease transmission.
As a result, it is the usual practice not to consider the recovered individuals as part of the problem formulation. We adopt this practice
throughout this paper.
 Various types of incidence rates $F(S,I)$ have been considered in the literature, for example,
the Holling type II functional response  $F(S,I)=\frac{\beta SI}{m_1+S}$ (see \cite{HUO}), the bilinear functional response $F(S,I)=\beta SI$ (see \cite{DANG,ZHANG}), the nonlinear functional response  $F(S,I)=\frac {\beta SI^l}{1+m_2I^h}$ (see \cite{RUAN,YANG}),
and the Beddington-DeAngelis functional response $F(S,I)=\frac{\beta SI}{1+m_1S+m_2I}$  (see \cite {NHU1, NHU2}).

It is by now widely known that in order
to have
a realistic model, one cannot ignore random environmental fluctuations (temperature, the climate, the water resources, etc.). In this paper we consider stochastic epidemic models of the form
\begin{equation}\label{e:unhid}
\begin{cases}
dS(t)=\big[a_1-b_1S(t)-I(t)f(S(t),I(t))\big]dt + \sigma_1 S(t) dB_1(t),\\
dI(t)=\big[-b_2I(t) + I(t)f(S(t),I(t))\big]dt + \sigma_2I(t) dB_2(t),\\
\end{cases}
\end{equation}
where $B_1(t)$ and $B_2(t)$ are independent Brownian motions and $\sigma_1,\sigma_2\neq 0$ are the noise intensities (standard deviations).
Moreover, in the above we have rewritten the coefficients: $b_1:=\mu_S, b_2:=\mu_I+r$ and $F(S,I)=If(S,I)$ (compare this to \eqref{1.0}). This system has been analyzed in a
general setting in \cite{DN-JDE}.

It is well-known that there are diseases for which certain infected individuals
are asymptomatic.
Covid-19 is one such example -- there have been many reports of infections where the infected exhibit no symptoms. We intend to capture this type of behavior in our model.
In order to do this, we assume that
the group of infected individuals that has incidence rate $I(t)f(S(t),I(t))$ (the rate that describes how the disease spreads from infected groups to susceptible groups) in the classical setting, now contains 2 sub-groups. The first group contains individuals who have been confirmed to be infected and
with the incidence rate $I(t)f_1(S(t),I(t))$ (we will still denote this incidence rate by $I(t)f(S(t),I(t))$
 for
  notational simplicity).
The second group has incidence rate $I(t)h(S(t),I(t))$ and contains  people whose
infection status is unknown or hidden.
Let $\alpha(t)$ be a Markov process
taking values in $\M=[0,1]$. We suppose that $\alpha(t)$
represents the percentage of
individuals in
the hidden-status class
that are infected at time $t\geq0$
and that only noise-corrupted observations of
$\alpha(t)$
are
available. More specifically, one can only observe $\alpha(t)$ with additive white noise.

It is natural to assume that the hidden status of potentially infected individuals affects the spread of the disease. As a result,
we let
 the functions $f$ and $h$
depend on $\alpha(t)$.
With
the hidden state dynamics in \eqref{e:unhid}, we obtain
\begin{equation}\label{eq-main-1}
\begin{cases}
dS(t)=\Big[a_1-b_1S(t)-I(t)f(\alpha(t),S(t),I(t))-\alpha(t)I(t)h(\alpha(t),S(t),I(t))\Big]dt + \sigma_1 S(t) dB_1(t),\\[1ex]
dI(t)=\Big[-b_2I(t) + I(t)f(\alpha(t),S(t),I(t))+\alpha(t)I(t)h(\alpha(t),S(t),I(t))\Big]dt + \sigma_2I(t) dB_2(t).
\end{cases}
\end{equation}

\begin{rem}
One can understand the dynamics by looking at individuals from group $S$, in which
susceptible individuals are infected at the rate $If(S,I)$.
We assume that $\alpha(t)$ percent of the rate of the potentially infected individuals are actually infected. Then we can say that members of the hidden group infect susceptible individuals at the rate $\alpha I h(S,I)$.
\end{rem}
\begin{rem}
We could combine $I(t)f(\alpha(t),S(t),I(t))+\alpha(t) I(t) h(\alpha(t),S(t),I(t))$ to
produce  a new function. However, we choose the current setup
 to make
the formulation and motivation clear. Moreover, this will also be more convenient for later discussions.
\end{rem}

Our results can be summarized as follows.
Because the infection status of certain individuals is hidden,
and
$\alpha(t)$ is not directly available,
the dynamics of \eqref{eq-main-1} are difficult to study.
To overcome the
difficulty, we apply the nonlinear filtering techniques by considering the conditional distribution of the process $\alpha(t)$ given the observations. This enables us to
replace the hidden Markov process $\alpha(t)$ in
\eqref{eq-main}
by
the corresponding conditional distribution.
We start by studying the well-posedness of the equation under consideration together with the positivity of solutions, the Markov-Feller property, and some moment estimates. Next, we study the longtime behavior of the system.
Under the assumption for ergodicity of nonlinear filtering  \cite{Han09,Kun71} and using ideas from dynamical systems, by considering the boundary equation and growth rate (see e.g., \cite{DANG,DN-JDE,HN18}), we are able to prove that
there is a threshold $\lambda$ such that if $\lambda<0$, the number of the infected individuals $I(t)$ tends to zero
exponentially fast
and if $\lambda>0$, all invariant probability measures of the system concentrate on $\mathbb{R}_+^{2,\circ}:=(0,\infty)^2$, and then the systems is permanent.
We show that the threshold $\lambda$ also characterizes
the permanence and extinction of the system \eqref{eq-main-1}.
 We also study the case when the process $\alpha(t)$ is a hidden Markov chain
 taking values in a finite set.
 Next, we demonstrate our results using
 simple examples and numerical simulations.

The rest of the paper is organized as follows. We give the mathematical formulation of our problem in Section \ref{secpre}.
Section \ref{sec:thres} introduces the threshold $\lambda$. The sign of $\lambda$ will be used to characterize the longtime behavior of the underlying system.
Section \ref{sec:lon} is devoted to the characterization of the longtime dynamics of our system. Section \ref{sec:disc} offers
some interpretations and implications of our results.
Finally, Section \ref{sec:exp} provides some simple examples and simulations to illustrate our theoretical results.

\section {
Problem
Formulation}\label{secpre}

Throughout this paper we use
$\R_+:=[0,\infty)$, $\R^\circ_+:=(0,\infty)$,
$\R_+^2:=[0,\infty) \times [0,\infty)$, and $\R^{2,\circ}_+:=(0,\infty)\times (0,\infty)$.
Let $(\Omega,\F,\{\F_t\}_{t\geq0},\PP)$ be a complete probability space with filtration $\{\F_t\}_{t\geq0}$ satisfying the usual conditions, and $B_1(t)$, $B_2(t)$, and $W(t)$ be mutually independent standard Brownian motions. The process
$\alpha(t)$ (termed a signal process) is assumed to be an adapted stochastic process taking values in $[0,1]$ that is independent of $B_1(t)$, $B_2(t)$, and $W(t)$.
Moreover, $\M$ will denote the space of probability measures on $([0,1],\mathcal B ([0,1]))$ endowed with the weak topology,
and $C[0,1]$ the spaces of all real-valued continuous functions on $[0,1]$.
For any function $l\in C[0,1]$ and $\mu\in\M$, set
$$
\mu(l):=\int_0^1 l(x)\mu(dx).
$$

As discussed above, we consider the setting where
the precise values
$\alpha(t)$ are not available and only noisy observations are available.
The observation process $y(t)$ of the signal process $\alpha(t)$
is given by
\begin{equation}\label{eq-oby}
dy(t)=g(\alpha(t))dt+dW(t),\quad y(0)=0,
\end{equation}
where $g:[0,1]\to\R$ is a continuous function.
Let
$
\F_t^y:=\sigma\{y(s):0\leq s\leq t\}\bigvee\sigma(\alpha(0)),$
where $\bigvee$ denotes the smallest $\sigma$-algebra generated by the union of some $\sigma$-algebras.
Let $\Pi_t(\cdot)\in \M$ be the conditional distribution of the signal process $\alpha(t)$ given the observation $y(t)$ and the initial data, i.e.,
$$
\Pi_t(A)=\PP[\alpha(t)\in A|\F_t^y], A\in\mathcal B([0,1]).
$$
Such $\{\Pi_t(\cdot)\}$ is called nonlinear filtering.


The field of nonlinear filtering has a long history.
The main idea stems from replacing  the unknown state by its conditional distributions. The earliest result was the well-known Kushner's equation
\cite{Kushner64}. Subsequently, the Duncan-Mortensen-Zakai equation came into being
\cite{Duncan,Mortensen,Zakai}. In this paper, we make use of the version of filtering developed by
Fujisaki-Kallianpur-Kunita \cite{FFK72}. We will also make use of the Wonham filter for hidden Markov chains, which is one of the handful finite-dimensional filters in existence \cite{Won65}.

To proceed,
we detail the results of
Fujisaki-Kallianpur-Kunita \cite{FFK72} (see also \cite{KS69}), which involve
a differential equation for the nonlinear filtering $\Pi_t(\cdot)$, next.
Define
$$
\beta(t)=y(t)-\int_0^t \Pi_s(g)ds,
$$
and note that the process $\beta(t)$ is a one-dimensional Wiener process, see e.g., \cite[Theorem 7.2]{LS01}.
Moreover, $
\sigma\{\beta(t_2)-\beta(t_2):t_2\geq t_1\geq t\}$ and $\F_t^y$ are independent for all $t\geq 0.$
If the signal process $\alpha(t)$ is a Markov process with infinitesimal generator $A$, then $\Pi_t(\cdot)$ is the solution of
\begin{equation}\label{eq-pi}
\Pi_t(l)=\Pi_0(l)+\int_0^t\Pi_s(Al)ds+\int_0^t \left( \Pi_s(lg)-\Pi_s(l)\Pi_s(g)\right)d\beta(s),\quad\forall l\in\mathcal D(A).
\end{equation}
The interested reader is referred to the detailed analysis given in \cite{FFK72,KS69}.

We will not make use of \eqref{eq-pi} often in our analysis, except for establishing some preliminary properties.
The stochastic differential equation for $\Pi_t(\cdot)$ is
rather complex and is not the main concern of the current paper.
As will be seen in the next section, we only need
to establish the related ergodicity. Thus,
for us, it suffices to consider
$\Pi_t(\cdot)$ as a stochastic process taking values in $\M$.
In addition,
we use
continuous measurable modification of $\Pi_t(\cdot)$;
such a modification always exists
\cite{Kun71}.

Under the premise that
one only observes a noisy version of
$\alpha(t)$,
we proceed to study  system \eqref{eq-main-1} by using
the nonlinear filter $\{\Pi_t(\cdot)\}$
with given the information of the observation process $y(t)$. More precisely, we consider the system
\begin{equation}\label{eq-main}
\begin{cases}
dS(t)=\Big[a_1-b_1S(t)-\int_0^1I(t)f(x,S(t),I(t))\Pi_t(dx)
\\
\hspace{2cm}
-I(t)\int_0^1xh(x,S(t),I(t))\Pi_t(dx)\Big]dt + \sigma_1 S(t) dB_1(t),\\[1ex]
dI(t)=\Big[-b_2I(t) + \int_0^1I(t)f(x,S(t),I(t))\Pi_t(dx)
\\
\hspace{2cm}
+I(t)\int_0^1xh(x,S(t),I(t))\Pi_t(dx)\Big]dt + \sigma_2I(t) dB_2(t),
\end{cases}
\end{equation}
where
$$
\begin{aligned}
&\Pi_t(A)=\PP[\alpha(t)\in A|\F_t^y],\forall A\in \mathcal B([0,1]),\\
&dy(t)=g(\alpha(t))dt+dW(t), \quad y(0)=0.
\end{aligned}
$$
Denote by $\PPi$ and $\Ei$ the probability and expectation corresponding to the initial values $S(0)=u$, $I(0)=v$,  $\Pi_0=\pi$, and the distribution of $\alpha(0)$, respectively.
We next make some assumptions that will be used throughout this paper.


\begin{asp}\label{asp-1}
	The following conditions hold:
	\begin{itemize}
		\item The function $f:[0,1]\times \R^2_+\to\R_+$ is nonnegative, $f(x,0,i)=0, \forall x\in[0,1]$, $i\geq 0$. Furthermore, $f$ is Lipschitz continuous,
		i.e., there exists a positive constant $L_1$ such that
		 for all $ x_1,x_2\in[0,1], s_1,s_2,i_1,i_2\geq 0$
		\begin{equation*}
		\begin{aligned}\abs{f(x_1,s_1,i_1)-f(x_2,s_2,i_2)}\leq L_1(|x_1-x_2|+\abs{s_1-s_2}+\abs{i_1-i_2}).
		\end{aligned}
		\end{equation*}
		\item The function $h:[0,1]\times\R_+^2\to\R_+$ satisfies $h(x,0,i)=0$, and is Lipschitz continuous with Lipschitz constant $L_2$, i.e., for all $s_1,s_2,i_1,i_2\geq 0, x_1,x_2\in[0,1]$,
				\begin{equation*}
		\begin{aligned}\abs{h(x_1,s_1,i_1)-h(x_2,s_2,i_2)}\leq L_2(|x_1-x_2|+\abs{s_1-s_2}+\abs{i_1-i_2}).
		\end{aligned}
		\end{equation*}
	\item For each $x\in[0,1]$, functions $h(x,\cdot,0)$ and $f(x,\cdot,0)$ are non-decreasing.
	\end{itemize}
\end{asp}

\begin{rem}
We note that almost all of the incidence rate functions used in the literature (such as the bilinear incidence rate, the Beddington-DeAngelis incidence rate, the Holling type II functional response, etc.) satisfy these conditions. \textit{Recall that the incidence rate in our setting is $If(S,I)$ rather than $f(S,I)$}.

The third condition is imposed because the incidence rate and the growth rate of the hidden class should increase when $S(t)$ and $I(t)$ increase. Since we have rewritten these rates as $If(S,I)$ and $Ih(\alpha,S,I)$, only increasing condition on $S$ is assumed.
\end{rem}

\begin{asp}\label{asp-2}
	The signal process $\alpha(t)$ is a Markov-Feller process that has a unique invariant measure $\mu^*$ and
		$$
		\|P(t,x,\cdot)-\mu^*(\cdot)\|_{\text{TV}}=0,
		$$
		where $P(t,x,\cdot)$ is the transition probability and $\|\cdot\|_{\text{TV}}$ is the total variation norm.
\end{asp}

\begin{rem}
	This assumption is needed to guarantee the ergodicity of the nonlinear filtering as discussed in Section \ref{sec:filter} later.
	Using this, we can then define a threshold that fully characterizes the longtime behavior of the underlying system; see Section \ref{sec:lambda}.
\end{rem}


For $V(s,i):\R^2\to\R$, define the operator $\Lom V$ by
$$
\begin{aligned}
\Lom V[s,i,\pi]=&\dfrac{\partial V}{\partial s} \Big[a_1-b_1s-if(s,i)-i\int_0^1xh(x,s,i)\pi(dx)\Big]\\
&+\dfrac{\partial V}{\partial i} \Big[-b_2i+if(s,i)+i\int_0^1xh(x,s,i)\pi(dx)\Big]\\
&+\frac{\sigma_1^2s^2}2\dfrac{\partial^2 V}{\partial s^2}+\frac{\sigma_2^2i^2}2\dfrac{\partial^2 V}{\partial i^2}.
\end{aligned}
$$
In the above, $(s,i,\pi)\in\R^2\times \M$ represents the variable of $\Lom V$ rather than that of $V$.

\para{Discrete state space and Wonham filter.}
If the Markov process $\alpha(t)$ takes values in a finite space $\{m_1,\dots,m_{n^*}\}\subset [0,1]$ and has generator $\{q_{ik}\}_{i,k\in\{1,\dots,n^*\}}$, the formulation will be
simpler and  more explicit.
We can formulate the problem as follows.
Let
$$
\begin{aligned}
& e_k(t):=\PP(\alpha(t)=m_k|\F_t^y)=\E [\1_{\{\alpha(t)=m_k\}}|\F_t^y], k=1,\dots,n^*,\\ &e(t)=(e_1(t),\dots,e_{n^*}(t)),\\
& \mathcal S_{n^*}:=\left\{e=(e_1,\dots,e_{n^*})\in\R^{n^*}:e_k\geq 0,\sum_{k=1}^{n^*}e_k=1\right\},\\
& g_k:=g(m_k), k=1,\dots,n^*, \quad \bar g(e):=\sum_{k=1}^{n^*} g_ke_k, e=(e_1,\dots,e_{n^*})\in\mathcal S_{n^*}.
\end{aligned}
$$
It was shown in \cite{Won65} that the posterior probability $e(t)$
satisfies the following system of stochastic differential equations
\begin{equation}
\begin{cases}
de_k(t)=\disp\!\left[\!\sum_{i=1}^{n^*} q_{ik}e_i(t)-(g_k-\bar g(e(t)))\bar g(e(t))e_k(t) \!\!\right]\!\!dt\!\!+\! (g_k-\bar g(e(t)))e_k(t)dy(t),\ k=1,\dots,n^*,\\
e_k(0)=e_k^0,\quad k=1,\dots,n^*.
\end{cases}
\end{equation}
In this case, instead of considering  system \eqref{eq-main}, one can study the following system of stochastic differential equation
\begin{equation}\label{eq-main-dis}
\begin{cases}
dS(t)=\Big[a_1-b_1S(t)-I(t)\sum_{k=1}^{n^*}f(m_k,S(t),I(t))e_k(t)
\\
\hspace{2cm}
-I(t)\sum_{k=1}^{n^*}m_kh(m_k,S(t),I(t))e_k(t)\Big]dt+ \sigma_1 S(t) dB_1(t),\\[1ex]
dI(t)=\Big[-b_2I(t) + I(t)\sum_{k=1}^{n^*}f(m_k,S(t),I(t))e_k(t)
\\
\hspace{2cm}
+I(t)\sum_{k=1}^{n^*}m_kh(m_k,S(t),I(t))e_k(t)\Big]dt + \sigma_2I(t) dB_2(t),\\
de_k(t)=\left[\sum_{i=1}^{n^*} q_{ik}e_i(t)-(g_k-\bar g(e(t)))\bar g(e(t))e_k(t) \right]dt+(g_k-\bar g(e(t)))e_k(t)dy(t),\quad  k=1,\dots,n^*.\\
\end{cases}
\end{equation}
Moreover, the process
$$
\bar W(t)=y(t)-\int_0^t \bar g(e(s))ds
$$
is a one-dimensional Brownian motion adapted to $\F_t^y$; see e.g., \cite[Theorem 7.2]{LS01}. Therefore,  system \eqref{eq-main-dis} can be rewritten as
\begin{equation}\label{eq-main-dis-1}
\begin{cases}
dS(t)=\Big[a_1-b_1S(t)-I(t)\sum_{k=1}^{n^*}f(m_k,S(t),I(t))e_k(t)
\\
\hspace{2cm}
-I(t)\sum_{k=1}^{n^*}m_kh(m_k,S(t),I(t))e_k(t)\Big]dt + \sigma_1 S(t) dB_1(t),\\[1ex]
dI(t)=\Big[-b_2I(t) + I(t)\sum_{k=1}^{n^*}f(m_k,S(t),I(t))e_k(t)
\\
\hspace{2cm}
+I(t)\sum_{k=1}^{n^*}m_kh(m_k,S(t),I(t))e_k(t)\Big]dt + \sigma_2I(t) dB_2(t),\\
de_k(t)=\sum_{i=1}^{n^*} q_{ik}e_i(t)+e_k(t)(g_k-\bar g(e(t)))d\bar W(t),\quad  k=1,\dots,n^*.\\
\end{cases}
\end{equation}
This system is easier to analyze than \eqref{eq-main}. However, in this case, we need to assume that the signal process $\alpha(t)$ representing the portion of the rate of the infection in the group of individuals with hidden infection status takes only finitely many values.
This significantly limits the setting as well as the possible applications to real world problems.

One may simplify the problem further by assuming that $\alpha(t)$ takes values in $\{0,1\}$. This would mean that at any given time all individuals in the hidden status group are either susceptible or infected.

\section{Ergodicty of Nonlinear Filter and Threshold for Permanence and Extinction}\label{sec:thres}
\subsection{Ergodicity of Nonlinear Filter}\label{sec:filter}
The study of the asymptotic properties of the nonlinear filter has a long history in the literature. We briefly summarize the developments. One of the first works is Kunita's paper \cite{Kun71}.
We restate the main result (Theorem 3.3) of this reference as follows.

\begin{pron}
$($Kunita 1971$)$
	Assume that the signal process $\alpha(t)$ taking values in a compact separable Hausdorff space is a Markov-Feller process with semigroup $P_t$ that has a unique invariant measure $\mu^*$ and
	$$
	\limsup_{t\to\infty}\int|P_tl(x)-\mu^*(l)|\mu(dx)=0,\forall l\in C([0,1]).
	$$
	Then the process $\Pi_t(\cdot)$ is an $\M$-valued Markov-Feller process that has unique invariant measure $\Phi^*$. Moreover, $\mu^*$ is the barycenter of $\Phi^*$, i.e.,
	$$
	\mu^*(l)=\int_{\M} \nu(l)\Phi^*(d\nu), \forall l\in C[0,1].
	$$
\end{pron}

Unfortunately, it was pointed out in \cite{BCL04} that there was a serious gap in the proof of the main result in \cite{Kun71}. A key role in the verification of the uniqueness for the invariant measure of $\Pi_t(\cdot)$ is the following identity
\begin{equation}\label{eq-identity}
\bigcap_{t\geq 0}\F_{[0,\infty)}^y\bigvee \sigma\{\alpha(s):s\geq t\}=\F_{[0,\infty)}^y\bigvee \left(\cap_{t\geq 0}\sigma\{\alpha(s):s\geq t\}\right),
\end{equation}
where $\F_{[0,\infty)}^y:=\bigvee_{t\geq 0}\F_t^y$.
This identity is indispensable in the proof of the uniqueness of the invariant measure of nonlinear filtering; see the counterexample given in \cite{BCL04}.
Moreover, the exchange of intersection and supremum is not always permitted in general\footnote{According to Williams \cite{Wil91}, this incorrect identity ``...tripped up even Kolmogorov and Wiener"; see \cite[p. 837]{Sin89}, and \cite[pp. 91--93]{Mas66}}.
However, in Kunita's proof, this identity was not proved.
On the other hand, it is important to note that all the known counterexamples are based on the degeneracy of the observation, i.e., there is no added noise.
Therefore, it was tempting to conjecture that the identity  \eqref{eq-identity} still holds provided the nondegeneracy of the observation.

In 2009, R. Handel \cite{Han09}  has
partially solved this open problem
 in a general setting.
In fact, \cite[Theorem 4.2]{Han09} proved that  identity \eqref{eq-identity} does indeed
hold under conditions of ergodicity of signal process \cite[Assumption 3.1]{Han09} and nondegeneracy of the observation process \cite[Assumption 3.2]{Han09}, which are only mildly stronger than those in \cite{Kun71}\footnote{According to Handel \cite{Han09}, whether Kunita's condition is already sufficient to guarantee uniqueness of the
	invariant measure with barycenter $\mu^*$ remains an open problem.}.
Finally, we state the following theorem on the ergodicity of the filter $\Pi_t(\cdot)$ under our setting and our assumption.

\begin{pron}
	Under Assumption \ref{asp-2}, the process $\Pi_t(\cdot)$ is an $\M$-valued Markov-Feller process and has a unique invariant measure $\Phi^*$. Moreover, $\mu^*$ is the barycenter of $\Phi^*$, i.e.,
	$$
	\mu^*(l)=\int_{\M} \nu(l)\Phi^*(d\nu), \forall l\in C[0,1].
	$$
\end{pron}

Moreover, let $\M_{\Phi^*}\subset\M$ be the support of the invariant measure $\Phi^*$ of the nonlinear filter $\Pi_t(\cdot)$. In general, one should not expect that $\M_{\Phi^*}=\M$. In fact, this does not hold even in the simple setting of the Wonham filter when the state space has only 3 states \cite[Section 4]{TW11}.

\subsection{Threshold for Permanence and Extinction}\label{sec:lambda}
We next use the ergodicity of the nonlinear filter developed in the previous section in conjunction with a Lyapunov exponent analysis (sometimes called invasion type analysis in population dynamics) coming from dynamical systems \cite{DANG,DN-JDE,HN18}. This allows us to introduce a threshold $\lambda$, which characterizes the longtime behavior of system \eqref{eq-main}.

Consider the equation on boundary when the infected individuals are absent, i.e.,
\begin{equation}\label{phi}
d \varphi(t)=\big(a_1-b_1\varphi(t)\big)dt+\sigma_1 \varphi(t)dB_1(t),\;\;\;\varphi(0)=u\geq 0.
\end{equation}
By solving the Fokker-Planck equation,   equation \eqref{phi}  has
a unique stationary distribution $\hat\mu$ with the density  given by
\begin{equation}\label{hpp}
\dfrac{b^a}{\Gamma(a)}y^{-(a+1)}e^{-\frac{b}y},\; y>0,
\end{equation}
where $c_1=b_1+\dfrac{\sigma_1^2}{2}, a=\dfrac{2c_1}{\sigma_1^2},b=\dfrac{2a_1}{\sigma_1^2} $ and
$\Gamma(\cdot) $ is the Gamma function.
The main idea is to determine whether $I(t)$ converges to 0 or not by
looking at the Lyapunov exponent $\limsup_{t\to\infty} \dfrac{\ln I(t)}t$ when $I(t)$ is small. Using It\^{o}'s formula yields
\begin{equation}\label{Lia}
\begin{aligned}
\dfrac{\ln I(t)}{t}=&\dfrac{\ln v}{t}+ \dfrac {\sigma_2B_2(t)}t-c_2+\dfrac 1{t}\int_0^t\int_0^1f(x,S(s),I(s))\Pi_s(dx)ds\\
&+\frac 1t\int_0^t \int_0^1xh(x,S(s),I(s))\Pi_s(dx)ds,
\end{aligned}
\end{equation}
where $c_2=b_2+\dfrac{\sigma_2^2}2.$
Intuitively, $\limsup_{t\to\infty} \frac{\ln I(t)}t<0$ implies $\lim_{t\to\infty} I(t)=0$. As a result, if $I(t)$ is small then $S(t)$ is close to $\varphi(t)$ provided $S(0)=\varphi(0)$. Therefore, when $t$ is sufficiently large we have
\begin{equation*}
\begin{aligned}
\dfrac 1{t}&\int_0^t\int_0^1f(x,S(s),I(s))\Pi_s(dx)ds+\frac 1t\int_0^t \int_0^1xh(x,S(s),I(s))\Pi_s(dx)ds\\
&\approx \dfrac 1{t}\int_0^t\int_0^1f(x,\varphi(s),0)\Pi_s(dx)ds+\frac 1t\int_0^t \int_0^1xh(x,\varphi(s),0)\Pi_s(dx)ds.
\end{aligned}
\end{equation*}
By the strong law of large numbers for $\varphi(t)$ and $\Pi_t$ from \eqref{Lia}, we obtain that the Lyapunov exponent of $I_{u,v}(t)$  can be
approximated by
\begin{equation}\label{ld}
-c_2+ \int_0^{\infty}\int_{\M} \left(f(x,y,0)\nu(dx)\right)\Phi^*(d\nu)\hat\mu(dy)+\int_0^\infty\int_{\M} \left(\int_0^1 xh(x,y,0)\nu(dx)\right)\Phi^*(d\nu)\hat\mu(dy).
\end{equation}
Since $\mu^*$ is the barycenter of $\Phi^*$, the Lyapunov exponent of $I(t)$ is approximated by
$$
-c_2+ \int_0^{\infty}\int_0^1f(x,y,0)\mu^*(dx)\hat\mu(dy)+\int_0^\infty\int_0^1 xh(x,y,0)\mu^*(dx)\hat\mu(y).
$$
Therefore, we define the threshold $\lambda$ by
\begin{equation}\label{eq-lambda}
\lambda:=-c_2+ \int_0^{\infty}\int_0^1f(x,y,0)\mu^*(dx)\hat\mu(dy)+\int_0^\infty\int_0^1 xh(x,y,0)\mu^*(dx)\hat\mu(dy).
\end{equation}
In the next section, we
prove that the sign of $\lambda$
 characterizes the longtime behavior of the system \eqref{eq-main}.
It is also noted that when $\alpha(t)$ is available,
so is \eqref{eq-main-1}, and the permanence or extinction of \eqref{eq-main-1} is also determined by the sign of $\lambda$ defined in \eqref{eq-lambda} (see e.g., \cite{DNY-SPA}).


\section{Characterization of Longtime Properties: Permanence and Extinction}\label{sec:lon}
\subsection{The existence and uniqueness of the solution and preliminary results}
We begin with the following theorem on the existence and uniqueness of the solution of \eqref{eq-main}
and then proceed with a complete characterization of its positivity and some other important properties.

\begin{thm}\label{thm-exi}
	For any $(u, v,\pi)\in\R_+^2\times\M$, there exists a unique global solution
to
system \eqref{eq-main} with initial value $(u,v,\pi)$.
	The three-component process $\{(S(t), I(t),\Pi_t), t\geq0\}$
	is a Markov process.
\end{thm}

\begin{proof}
	We prove the existence and uniqueness of the solution of \eqref{eq-main} first.
It is noted that although we have assumed $f(s,i)$ is Lipschitz continuous, the coefficient $if(s,i)$ in the system \eqref{eq-main} is not globally Lipschitz in general.
Since the coefficients of the equation are locally Lipschitz continuous, there is a unique solution $(S(t),I(t))$ with the  initial value $(u,v,\pi)\in \R_+^{2}\times\M$, defined on maximal interval  $t\in [0,\tau_e)$, $\tau_e:=\inf\{t\geq 0: S(t)\vee I(t)=\infty\}$ with the convention $\inf\emptyset=\infty$; see e.g., \cite[Theorem 3.8 and Remark 3.10]{MAO06}.
We need to show $\tau_e=\infty$ a.s.
   If we define
$$\tau_k=\inf\Big\{t\geq 0: S(t)\vee I(t)>k\Big\},$$
then $\tau_e=\lim_{k\to\infty}\tau_k$.
Consider $V_1(s,i)=s+i$, then
 we have from definition of $\Lom V_1$
 that
$$
\begin{aligned}
\Lom V_1(s,i,\pi)&=a_1-b_1s-b_2i
\leq a_1\;\forall (s,i,\pi)\in\R^2_+\times\M.
\end{aligned}
$$
Hence, by applying It\^o's formula and taking expectation, we obtain
$$
\begin{aligned}
\Ei V_1\big(S(\tau_k\wedge t), I(\tau_k\wedge t)\big)
\leq V_1(u,v)+a_1t,
\end{aligned}
$$
which together with Markov's inequality implies that
$$
\begin{aligned}
\PPi\big\{\tau_k<t\big\}&\leq \PPi\Big\{ V_1\big(S(\tau_k\wedge t), I(\tau_k\wedge t),\alpha(\tau_k\wedge t)\big)\geq k\Big\}
\\&\leq \dfrac{ V_1(u,v)+a_1t}{k}\to 0\;\text{as}\;k\to\infty.
\end{aligned}
$$
Therefore, we have $\PPi\{\tau_e\leq t\}=0$ or $\PPi\{\tau_e>t\}=1$ for all $t>0$. As a consequence, $\PPi\{\tau_e=\infty\}=1$. Hence,  system \eqref{eq-main} has a unique, global, and continuous solution.

We proceed
 to prove the Markov property. Since $\Pi_t$ is a Markov process and is independent of $B_1(t)$  and $B_2(t)$, the Markov property of the joint process $(S(t),I(t),\Pi_t)$ follows by standard arguments; see
for example,
\cite[Theorem 3.27 and Lemma 3.2]{MAO06} or \cite[Lemma 6.1]{FFK72}. To see why the argument in \cite{MAO06} can be applied,
note that $\Pi_t$ satisfies the stochastic equation \eqref{eq-pi} driven by $\beta(t)$ and the $\sigma$-algebra generated by increments $\{\beta(t_2)-\beta(t_1):t_2\geq t_1\geq t\}$ is independent of $\F_t^y$.
\end{proof}

Next, using Lyapunov functions, we estimate the moments of $S(t)$  and $I(t)$, and obtain some related results. Define $\sigma_*^2:=\max\{\sigma_1^2,\sigma_2^2\}$.

\begin{lem}\label{lem2.1}
The following assertions hold:
\begin{itemize}
\item[{\rm(i)}] For any $0<p<\frac{2\kappa}{\sigma_*^2}$ there is a constant $Q_1$ such that
$$\limsup\limits_{t\to\infty}\Ei (S(t)+ I(t))^{1+p}\leq Q_1\;\; \forall (u,v,\pi) \in \R_+^2\times\M.$$
\item[{\rm(ii)}] For any $\eps>0$, $H>1$, $T>0$, there is $\bar H =\bar H(\eps, H, T)$ such that
$$\PPi \left\{\dfrac{1}{\bar H}  \leq S(t)
 \leq \bar H, \;\;\forall t \in [0,T] \right\}\geq 1-\eps \;\; \text{if}\;\;\; (u,v,\pi) \in [H^{-1},H]\times [0;H]\times\M,$$
and
$$\PPi\{ 0 \leq S(t), I(t) \leq \bar H, \;\;\forall t \in [0,T] \}\geq 1-\eps \;\; \text{if}\;\;\; (u,v,\pi) \in [0,H]\times [0;H]\times\M.$$
\end{itemize}
\end{lem}

\begin{proof}
Consider the Lyapunov function
$V_3(s,i)=(s+i)^{1+p}.$
By directly calculation with
the differential operator $\Lom V_3$ and using Assumption \ref{asp-1}, we obtain
\begin{align*}
\Lom V_3(s,i,\pi)=&(1+p)(s+i)^p(a_1-b_1s-b_2i)
+\frac {p(1+p)}2(s+i)^{p-1}\Big( \sigma_1^2s^2+\sigma_2^2i^2\Big)\\
\leq &-(1+p)(s+i)^{p-1}\Big[ -\frac p2 \sigma_*^2(s+i)^2-a_1(s+i)\big],\;\forall (s,i,\pi)\in\R^2_+\times\M.
\end{align*}
Let
$0<C_4<\frac{p(1+p)}2\sigma_*^2$. By some standard calculations, we get
$$C_5=\sup_{(s,i,\pi)\in \R^2_+\times\M}\left\{\Lom V_3(s,i,\pi)+C_4V_3(s,i)\right\}<\infty.$$
This implies
\begin{equation}\label{10}
\Lom V_3\leq C_5-C_4 V_3.
\end{equation}
Applying \cite[Theorem 5.2, p.157]{MAO} proves part (i) of the lemma. The proof of part (ii) follows
from
part (i) and standard arguments;
see \cite[Lemma 2.1]{DN-JDE}.

\end{proof}

\begin{thm}\label{thm-markov}
	The
process $(S(t),I(t),\Pi_t)$ is a strong Markov and Feller process.
	Moreover, we have
	$\PPi\{S(t)>0, t>0\}=1$ and
	$\PP_{u,0,\pi}\{I(t)=0, t>0\}=1$,
	$\PPi\{I(t)>0, t>0\}=1$ provided $v>0$.
\end{thm}

\begin{proof}
It is easily seen that
the solution of \eqref{eq-main} is a homogeneous strong Markov and Feller process provided that the coefficients are globally
Lipschitz; see e.g., \cite[Theorem 2.9.3]{MAO} and \cite[Section 2.5]{YIN}.
It is noted that the space $\M$ of probability measures in $[0,1]$ endowed with the weak topology can be metricized by the bounded Lipschitz metric defined by
$$
\|\pi_1-\pi_2\|_{\text{BL}}:=\sup\left\{|\pi_1(l)-\pi_2(l)|: \|l\|\leq 1, \sup_{x\neq y\in[0,1]}\frac{|l(x)-l(y)|}{|x-y|}\leq 1\right\}.
$$

Therefore, by using the results in Lemma \ref{lem2.1}, we obtain from the local Lipschitz property of coefficients of \eqref{eq-main} and a truncation
argument that $(S(t),I(t),\Pi_t)$ is a homogeneous strong Markov and Feller process.
The details of this truncated argument and this result can be found in \cite[Theorem 5.1]{NYZ17}.

Next, we establish the positivity of solutions. First, suppose that $u,v>0$.
Let us consider  the Lyapunov function $V_2: \R_+^{2}\rightarrow \R_+$
$$V_2(s,i)=\left(s-1-\ln s\right)+(i-1-\ln i).
$$
By direct calculations, we have
\begin{equation*}
\begin{aligned}
\Lom V_2(s,i,\pi)&=\left(1-\frac {1}{s}\right)\left(a_1-b_1s-if(x,s,i)-i\int_0^1xh(x,s,i)\pi(dx)\right)+ \frac {\sigma_1^2s^2}{2s^2}
\\
&\;\;\;+\left(1-\frac 1{i}\right)\left(-b_2i+if(x,s,i)+i\int_0^1xh(x,s,i)\pi(dx)\right) +\frac {\sigma_2^2i^2}{2i^2}.
\end{aligned}
\end{equation*}
It follows from Assumption \ref{asp-1} that
$f(x,s,i)=\abs{f(x,s,i)-f(x,0,i)}\leq L_1s$ and
$h(x,s,i)=\abs{h(x,s,i)-h(x,0,i)}\leq L_2s$.
Therefore, it is easily seen that
\begin{equation*}
\begin{aligned}
\Lom V_2(s,i)
&\leq C_1+\frac{f(s,i)i}{s} +\frac{i\int_0^1 h(x,s,i)\pi(dx)}s
\\&\leq C_1+(L_1+L_2)(s+i),
\end{aligned}
\end{equation*}
where $C_1= a_1+ b_1+\frac{\sigma_1^2}2+b_2+\frac{\sigma_2^2}2.$
As a result, if we let $C_2=L_1+L_2+1$ and $C_3=C_1+2C_2\ln C_2+2C_2$ then
\begin{equation}\label{eq-V2}
\begin{aligned}
\Lom V_2(s,i,\pi)-C_2V(s,i)&\leq C_1-s-i+C_2(\ln s+\ln i)+2C_2\\
&\leq C_1+2C_2\ln C_2+2C_2=C_3.
\end{aligned}
\end{equation}
For $k>1$, denote
$$\eta=\inf\big\{t\geq 0: S(t)\wedge I(t)\leq 0\big\},$$
$$\eta_k=\inf\Big\{t\geq 0: S(t)\wedge I(t)<\dfrac 1k\Big\}.$$
Then $\eta=\lim_{k\to\infty}\eta_k$.
Therefore, by using the same argument as above, we obtain from \eqref{eq-V2} that
\begin{equation*}
\begin{aligned}
\PPi\{\eta_k<t\}\leq \dfrac{ V_2(u,v)+C_3t}{e^{-C_2t}(\ln k-1)}\to 0\;\text{as}\;k\to \infty.
\end{aligned}
\end{equation*}
As a result, $\PPi\{\eta_\infty=\infty\}=1$. This implies that
\begin{equation}\label{e2-thm2.1}
\PPi\big\{S(t)>0: t>0 \big\}=\PPi\big\{I(t)>0: t>0 \big\}=1\;\forall u,v>0.
\end{equation}
If $u>0,v=0$, the result $\PPi\big\{S(t)>0: t>0 \big\}=1$ can be shown similarly. Moreover, it is obvious that $\PP_{u,0,\pi}\big\{I(t)=0: t>0 \big\}=1$.

We are in a position
to consider the case $u=0$ and $ v\geq 0$ and prove the positivity of $S(t)$.
Let $\eps>0$ be sufficiently small such that
\begin{equation}\label{e3-thm2.1}a_1-b_1\tilde u-\tilde v\sup_{x\in[0,1]}\big(f(x,\tilde u,\tilde v)+h(x,\tilde u,\tilde v)\big)\geq \frac{a_1}{2},
\end{equation}
for any $(\tilde u,\tilde v)\in\R^2$ satisfying
$\tilde u+|\tilde v-v|<\eps$.
Such an $\eps$ exists due to Assumption \ref{asp-1}.
Set $$\tilde\tau_1=\inf\{t>0:S(t)+|I(t)-v|\geq\eps\}.$$
By the continuity of $(S(t), I(t))$ it is clear that $\PP_{0,v,\pi}\{\tilde\tau_1>0\}=1$.
It follows
from \eqref{e3-thm2.1} that
$$a_1-b_1S(t)-I(t)\int_0^1f(x,S(t),I(t))\Pi_t(dx)+I(t)\int_0^1xh(x,S(t),I(t))\Pi_t(dx)>0\;\text{if}\;t\in(0,\tilde\tau_1].$$
This and the variation of constants formula (see \cite[Chapter 3]{MAO})
imply that
$$\PP_{0,v,\pi}\{S(t)>0, t\in(0,\tilde\tau_1]\}=1,$$
which combined with \eqref{e2-thm2.1}
and the strong Markov property of $(S(t),X(t),\Pi(t))$
yields that
\begin{equation*}
\PP_{0,v,\pi}\{S(t)>0, t\in(0,\infty)\}=1.
\end{equation*}
The theorem is therefore proved.
\end{proof}

\subsection{Extinction}\label{secext}
Consider the case
$\lambda<0$.
We shall show that the number of the infected individuals $I(t)$ tends to zero at an exponential rate while the number of the susceptible individuals $S(t)$ converges to $\varphi(t)$.

\begin{thm}\label{R<1}
Assume that $\lambda<0$. Then for any initial point $(u,v,\pi)\in\mathbb{R}_+^{2,\circ}\times\M$,
the number of the infected individuals $I(t)$ tends to zero  at an exponential rate, i.e., 
$$\PPi\left\{\limsup\limits_{t\to\infty}\dfrac{\ln I(t)}{t}= \lambda\right\}=1,$$
and the susceptible class $S(t)$ converges weakly to the solution $\varphi(t)$ on the boundary.
\end{thm}

In order to prove Theorem \ref{R<1}, we need the following auxiliary results.

\begin{lm}\label{lem2.2}
For any $T,H>1,$ $\eps >0, \theta >0$, there is a $\delta =\delta (H,T,\eps,\theta)$ such that
$$\PPi\{ \tau_{\theta} \geq T \}\geq 1-\eps,\;\;\forall \; (u,v,\pi) \in [0,H]\times (0,\delta]\times\M,$$
where $\tau_{\theta}=\inf\{t\geq 0: I_{u,v}(t)> \theta \}.$
\end{lm}

\begin{proof}
By the exponential martingale inequality \cite[Theorem 7.4, p. 44]{MAO}, we have $\PPi(\Omega_1)\geq 1-\dfrac{\eps}{2}$, where
$$\Omega_1=\left\{ \sigma_2 B_2(t)\leq \dfrac{\sigma_2^2t}{2}+\ln \dfrac{2}{\eps}\;\forall t\geq 0\right\}.$$
In view of part (ii) Lemma \ref{lem2.1}, there exists a $\bar H=\bar H(T,H,\eps)$ such that $\PPi(\Omega_2)\geq 1-\dfrac{\eps}{2}$, where
$$\Omega_2=\{ 0 \leq S(t),I(t) \leq \bar H \;\;\forall t \in [0,T] \}.$$
Applying  It\^{o}'s formula to equation \eqref{eq-main} yields
that
\begin{equation}\label{lni}
\begin{aligned}
\ln I(t)=&\ln i -c_2t+\int_0^t\int_0^1f(x,S_{u,v}(s),I_{u,v}(s))\Pi_s(dx)ds\\  &+\int_0^t\int_0^1xh(x,S(s),I(s))\Pi_s(dx)ds
+\sigma_2B_2(t).
\end{aligned}
 \end{equation}
Therefore, for any $(s,i,\pi) \in [0,H]\times (0,H]\times\M$ and $\omega \in \Omega_1 \cap \Omega_2$ we have from \eqref{lni} and the Lipschitz continuity of $f$ and $h$ that
$$\ln I(t) < \ln i - b_2T+T\big(2L_1\bar H+2L_2\bar H\big) + \ln \dfrac{2}{\eps},\;\;\forall t\in [0,T]. $$
Hence, we can choose a sufficiently small $\delta =\delta (H,T,\eps,\theta)<H$  such that for all $(s,i,\pi)\in [0,H]\times (0,\delta]\times\M$ and $0\leq t\leq T$, $\ln I(t)<\ln\theta,\forall\omega\in\Omega_1\cap\Omega_2$. The proof is complete.
\end{proof}

\begin{prop}\label{lm3.1}
Suppose that the assumptions from Theorem \ref{R<1} hold. For any $0<\varepsilon<\min\{\frac 15,-\frac{\lambda}{5}\}$ and $H>1$, there exists $\widehat \delta=\widehat \delta(\eps, H)\in (0,H^{-1})$ such that
$$\PPi\left\{\limsup_{t\to\infty}\dfrac{\ln I(t)}{t}=\lambda  \right\}\geq 1-4\eps,\;\forall (u,v,\pi)\in [H^{-1};H]\times(0;\widehat \delta]\times\M_{\Phi^*}.$$
\end{prop}

\begin{proof}
Let $\theta_0=\theta_0(\eps)<\frac{\eps}{L_1+L_2}\wedge \frac {b_1}{L_1+L_2}$ be such that
\begin{equation}\label{eq-theta0}
\frac{a_1+\theta_0(L_1+L_2+L_1\theta_0+L_2\theta_0)}{b_1-(L_1+L_2)\theta_0}-\frac{a_1}{b_1}<\frac{\eps}{L_1+L_2}.
\end{equation}
Consider the following stochastic differential equation
\begin{equation}\label{eq-bphi}
d \bar\varphi(t)=\big(\bar a_1-\bar b_1\bar \varphi(t)\big)dt+\sigma_1 \bar \varphi(t)dB_1(t),
\end{equation}
where $\bar a_1=a_1+L_1+L_2+L_1\theta_0+L_2\theta_0$, $\bar b_1=b_1-(L_1+L_2)\theta_0>0$.
A comparison result shows that $\varphi(t)\leq\bar\varphi(t),t\geq 0\a.s$ provided that $\varphi(0)=\bar \varphi(0)$.
Moreover, the strong law of large numbers yields that
\begin{equation}\label{eq-phibphi}
\begin{aligned}
\lim_{t\to\infty}\frac 1t&\bigg|\int_0^t\int_0^1f(x,\bar\varphi(s),0)\Pi_s(dx)ds+\int_0^t\int_0^1xh(x,\bar \varphi(s),0)\Pi_s(dx)ds\\
&\quad-\int_0^t\int_0^1f(x,\varphi(s),0)\Pi_s(dx)ds-\int_0^t\int_0^1xh(x, \varphi(s),0)\Pi_s(dx)ds\bigg|\\
&\leq \lim_{t\to\infty} \frac1t\int_0^t( L_1+L_2)(\bar \varphi(s)-\varphi(s))ds\\
&=(L_1+L_2)\left(\frac{\bar a_1}{\bar b_1}-\frac{a_1}{b_1}\right)\a.s
\end{aligned}
\end{equation}
Combining \eqref{eq-theta0} and \eqref{eq-phibphi} implies that
\begin{equation}\label{eq-phibphi-1}
\begin{aligned}
\lim_{t\to\infty}\frac 1t&\bigg|\int_0^t\int_0^1f(x,\bar\varphi(s),0)\Pi_s(dx)ds+\int_0^1xh(x,\bar \varphi(s),0)\Pi_s(dx)ds
\\
&\quad-\int_0^t\int_0^1f(x,\varphi(s),0)\Pi_s(dx)ds-\int_0^t\int_0^1xh(x, \varphi(s),0)\Pi_s(dx)ds\bigg| <\eps\a.s
\end{aligned}
\end{equation}
Therefore, there exists $T_1=T_1(\eps)$ such that
$\PP(\Omega_3)\geq 1-\eps$, where
$$
\begin{aligned}
\Omega_3:=&\Bigg\{\bigg|\frac 1t\int_0^t\int_0^1f(x,\bar\varphi(s),0)\Pi_s(dx)ds+\frac1t\int_0^t\int_0^1xh(x,\bar \varphi(s),0)\Pi_s(dx)ds\\
&\quad-\frac 1t\int_0^t\int_0^1f(x,\varphi(s),0)\Pi_s(dx)ds-
\frac1t\int_0^t\int_0^1xh(x, \varphi(s),0)\Pi_s(dx)ds\bigg| <2\eps,\forall t\geq T_1\Bigg\}
\end{aligned}
$$
By definition of $\lambda$ and ergodicity of $\varphi(t), \Pi_t$, we obtain
\begin{equation*}
\PP_{H,\pi}\left\{-c_2+ \lim_{t\to\infty}\dfrac1t\int_0^{t}\int_0^1f(x,\varphi(s),0)\Pi_s(dx)ds+\lim_{t\to\infty}\frac 1t\int_0^t \int_0^1xh(x,\varphi(s),0)\Pi_s(dx)ds= \lambda \right\}=1,
\end{equation*}
where $\PP_{H,\pi}$ indicates the initial values of $(\varphi(t),\Pi_t)$.
As a result, there exists $T_2=T_2(H,\eps)>1$ such that $\PP_{H,\pi}(\Omega_4)\geq 1-\eps$, where
$$\Omega_4=\left\{-c_2+\dfrac1t \int_0^t\int_0^1 f(x,\varphi(s),0)\Pi_s(dx)ds+\frac 1t\int_0^t \int_0^1xh(x,\varphi(s),0)\Pi_s(dx)ds\leq \lambda + \eps,\;\forall t\geq T_2\right\} .$$
In view of the uniqueness
of solutions, we have for all $u\in[0,H]$ that $\varphi_u(s)\leq \varphi_H(s), s \geq 0$ almost surely where the subscript of $\varphi(s)$ indicates the initial value $\varphi(0)$.
This implies that $\PP_{u,\pi}(\Omega_4)\geq 1-\eps$ for all $(u,\pi)\in[0,H]\times\M_{\Phi^*}$.

Since $\lim\limits_{t\to\infty}\dfrac{B_2(t)}{t}=0 \text{ a.s.},$ there is a $T_3=T_3(\eps)>1$ such that $\mathbb{P}(\Omega_5)\geq1-\eps$ where
$$\Omega_5=\left\{\sigma_2\dfrac{B_2(t)}{t}<\eps\;\forall t \geq T_3\right\}.$$
Let $T=\max\{T_1,T_2,T_3\}$. By Lemma \ref{lem2.2}, there is a $\hat\delta=\hat\delta(\eps,H)<\theta_0$ such that for all $(u,v,\pi)\in[0,H]\times (0,\hat\delta]\times\M$,
$
\PPi(\Omega_6)\geq 1-\eps,
$
where
$$
\Omega_6=\{\tau_{\theta_0}\geq T\}.
$$

Now, it follows from \eqref{lni} that $\forall (u,v,\pi) \in [0,H]\times (0,\delta]\times\M_{\Phi^*}$  we have in $\bigcap_{j=3}^{6}\Omega_j$, for all $t\in [T,\tau_{\theta_0}]$ that
\begin{equation}\label{lni2}
\begin{aligned}
\ln I(t)
=&\ln v -c_2t +  \!\!\int_0^{t}\!\Big(\!\int_0^1\! f(x,S(s),I(s))\Pi_s(dx) + \!\! \int_0^1\! xh(x,S(s),I(s))\Pi_s(dx)\!\Big)ds+\sigma_2B_2(t)\\
=&\ln v -c_2t +\int_0^{t}\left(\int_0^1 f(x,S(s),0)\Pi_s(dx) + \int_0^1xh(x,S(s),0)\Pi_s(dx)\right)ds+\sigma_2B_2(t)
\\&+ \int_0^{t}\bigg (\int_0^1f(x,S(s),I(s))\Pi_s(dx) + \int_0^1xh(x,S(s),I(s))\Pi_s(dx)
\\
&\qquad\qquad-\int_0^1f(x,S(s),0)\Pi_s(dx) - \int_0^1xh(x,S(s),0)\Pi_s(dx)\bigg)ds\\
\leq&\ln v -c_2t +  \int_0^{t}\left( \int_0^1f(x,\bar \varphi(s),0)\Pi_s(dx) + \int_0^1xh(x,\bar \varphi(s),0)\Pi_s(dx)\right)ds\\
&\qquad +\eps t+(L_1+L_2)\int_0^t I(s)ds\\
\leq&\ln v -c_2t +  \int_0^{t}\left( \int_0^1f(x,\varphi(s),0)\Pi_s(dx) + \int_0^1xh(x,\varphi(s),0)\Pi_s(dx)\right)ds+3\eps t\\
\leq&\ln v +(\lambda+4\eps)t
<\ln\hat\delta<\ln\theta_0.
\end{aligned}
\end{equation}
In the above, we have used the fact of that whenever $I(t)\leq \theta_0$, one has
$$
\begin{aligned}
I(s)\int_0^1xh(x,S(s),I(s))\Pi_s(dx)&\leq \theta_0(h(0,0,0)+L_2+L_2S(s)+L_2I(s))\\
&\leq \theta_0(L_2+L_2\theta_0)+L_2\theta_0S(s),
\end{aligned}
$$
so $S(s)\leq\bar\varphi(s)$ for all $t\in [0,\tau_{\theta_0}]$.

As a result of \eqref{lni2}, we must have $\tau_{\theta_0}=\infty$ for $\omega\in\bigcap_{j=3}^6\Omega_j$.
We obtain this claim by a contradiction argument as follows. If the claim is false then we have a set $\Omega_7\subset\bigcap_{i=3}^6\Omega_i$
with $\PP(\Omega_7) > 0$ and $\tau_{\theta_0} <\infty$ for any $\omega\in\Omega_7$. We already proved that $\tau_{\theta_0}>T$ for
$\omega\in\bigcap_{i=3}^6\Omega_i$. Moreover, in view of \eqref{lni2}, we have $I(t) \leq \hat\delta<\theta_0$ for any $t\in[T,\tau_{\theta_0}]$. Because
$I(t)$ is continuous almost surely, for almost all $\omega\in\Omega_7$ we have that $\lim_{t\to\tau_{\theta_0}}I(t) = I(\tau_{\theta_0}) <\hat\delta<\theta_0$,
which is a contradiction. So, $\tau_{\theta_0}=\infty$ for $\omega\in\bigcap_{j=1}^4\Omega_j$.
We deduce from $\tau_{\theta_0}=\infty$ and \eqref{lni2} that
$\lim_{t\to\infty} I(t)=0$ for almost $\omega\in\bigcap_{j=3}^6\Omega_j$.

Next, because $I(t)\leq \theta_0$ for any $t\geq 0$ for almost all $\omega	\in\bigcap_{j=3}^6\Omega_j$, we have shown that $S(t)\leq \bar\varphi(t), \forall t\geq 0$ almost surely in $\bigcap_{j=3}^6\Omega_j$.
Similar to \eqref{phi}, since $\bar b_1>0$,
the solution to \eqref{eq-bphi} has a unique invariant measure, say $\bar\mu$.
Then we have from the ergodicity of $\bar \varphi(t)$ that for some small $\hat p>0$,
\begin{equation}\label{se6}
\limsup_{t\to\infty}\frac1t\int_0^t S^{1+\hat p}(u)du\leq \lim_{t\to\infty}\frac1t\int_0^t \bar\varphi^{1+\hat p}(u)du=\int x^{1+\hat p}\bar\mu(dx)<\infty \text{ almost surely in }\bigcap_{j=3}^6\Omega_j.
\end{equation} 	
Using \eqref{se6}, the fact that $\lim_{t\to\infty} I(t)=0$, and the compactness of $\M$,
the family of random occupation measures
$$\wdt U^t_{u,v,\pi}(\cdot):=\frac 1t\int_0^t \1_{\{(S(s),I(s),\Pi_s)\in \cdot\}}ds$$ is tight for almost all $\omega\in\bigcap_{j=1}^3\Omega_j$.
From \cite[Lemma 5.6]{HN18},
with probability 1, any weak limit of $\wdt U^t_{u,v,\pi}(\cdot)$ as $t\to\infty$ (if it exists)
is an invariant probability measure,
which has support on $[0,\infty)\times\{0\}\times\M.$	
Because $\hat\mu\times\bdelta\times\Phi^*$, where $\bdelta$ is the Dirac measure concentrated at $0$, is an
invariant probability measure on $[0,\infty)\times\{0\}\times\M$, the family
$\wdt U^t_{u,v,\pi}(\cdot)$ converges weakly to $\bmu_0\times\bdelta$ almost surely in $ \bigcap_{j=3}^6\Omega_j$ as $t$ tends to $ \infty$.
One has from the weak
convergence
and the uniform integrability
in
\eqref{se6} that
$$
\begin{aligned}
\lim_{t\to\infty}\frac{\ln I(t)}t=&\lim_{t\to\infty}\frac1t\bigg[
-b_2t-\frac{\sigma_2^2t}2-\int_0^t\int_0^1f(x,S(s),I(s))\Pi_s(dx)ds \\
&\qquad\qquad +\int_0^t\int_0^1xh(x,S(s),I(s))\Pi_s(ds)ds\bigg]+\lim_{t\to\infty}\frac{\sigma_2B_2(t)}t\\
=&\lim_{t\to\infty}\!\int_{\R^2_+\times\mathcal P(\M)}\!\!\Big[\!-c_2 +\int_0^1f(x,y,i)\nu(dx)+\int_0^1 xh(x,y,i)\nu(dx)\!\Big]\wdt U^t_{u,v,\pi}(dy,di,d\nu)\\
=&-c_2 +\int_0^\infty \!\int_\M\!\Big[\int_0^1f(x,y,0)\nu(dx)\!\Big]\Phi^*(d\nu)\hat\mu(dy)\\
&+\int_0^\infty\int_{\M} \Big[\int_0^1 xh(x,y,0)\nu(dx)\Big]\Phi^*(d\nu)\hat\mu(dy)\\
=&\lambda<0,
\end{aligned}
$$
for almost every $\omega\in \bigcap_{j=3}^6\Omega_j$, $(u,v ,\pi)\in [0,H]\times (0,\hat\delta]\times\M_{\Phi^*}$.
The proof is complete by noting that $\PP(\bigcap_{j=3}^6\Omega_j)>1-4\eps$.
\end{proof}

\begin{proof}[Proof of Theorem \ref{R<1}]
	Let $\eps>0$ be arbitrary.
	In view of Proposition \ref{lm3.1},
	the process $(S(t), I(t))$ is transient (see e.g., \cite{WK} for definition) in $\R^{2,\circ}_+$.
	Thus, the process has no invariant probability measure in $\R^{2,\circ}_+$.
	Thus $\hat\mu\times\bdelta\times\Phi^*$
	is the unique invariant probability measure of $(S(t), I(t), \Pi_t)$   in $\R^2_+\times\M_{\Phi^*}$.
	
	Let $H$ be sufficiently large that $\hat\mu((0,H))>1-\eps$.
	Thanks to Lemma \ref{lem2.1} part (i) and compactness of $\M$, the process  $(S(t), I(t), \Pi_t)$ is tight. Consequently,
	the occupation measure
	$$U^t_{u,v,\pi}(\cdot):=\dfrac1t\int_0^t\PPi\left\{(S(s),I(s),\Pi_s)\in\cdot\right\}ds$$
	is tight in $\R^2_+\times\M$.
	Since any weak limit of $U^t_{u,v,\pi}$ as $t\to\infty$ must be an invariant probability measure of $(S(t),I(t),\Pi_t)$ (see \cite{HN18}),
	we have that
	$U^t_{u,v,\pi}$ converges weakly to $\hat\mu\times\bdelta\times\Phi^*$ as $t\to\infty$.
	As a result, for any $\delta>0$, there exists a  $\hat T>0$ such that
	$$U^{\hat T}_{u,v,\pi}((0,H)\times(0,\delta)\times\M_{\Phi^*})>1-\eps,$$
	or equivalently,
	$$\dfrac1{\hat T}\int_0^{\hat T}\PPi\{(S(t),I(t),\Pi_t)\in (0,H)\times(0,\delta)\times \M_{\Phi^*}\}dt>1-\eps.$$
As a result, we have
	$$\PPi\{\hat\tau\leq\hat T\}>1-\eps,$$
	where $\hat\tau=\inf\{t\geq 0: (S(t),I(t),\Pi_t)\in (0,H)\times(0,\delta)\times \M_{\Phi^*}\}$.
	Using the strong Markov property and Proposition \ref{lm3.1},
	we have that
	$$
	\PPi\left\{\lim_{t\to\infty}\dfrac{\ln I(t)}t\leq \lambda+4\eps\right\}\geq 1-\eps,
	$$
	for any $(u,v,\pi)\in\R^{2,*}_+\times\M$. Therefore, since $\eps > 0$ is arbitrary, the assertion in convergence of $I(t)$ follows.
Once we have the exponentially fast convergence to $0$ of $I(t)$, the convergence of $S(t)$ to $\varphi(t)$ follows from standard arguments; see, for example, \cite{DANG,DN-JDE}.
\end{proof}

\subsection {Permanence}\label{secper}
In this section, we deal with  the case $\lambda>0$ and prove that the system is permanent in the following sense.
\begin{deff}
	We say that system \eqref{eq-main} is permanent (in mean) if for any initial value $(u,v,\pi)\in\R_+^{2,\circ}\times\M$
	$$
	\liminf_{t\to\infty}\frac 1t\int_0^t \E_{u,v,\pi}S(s)ds>0,\quad \liminf_{t\to\infty}\frac 1t\int_0^t \E_{u,v,\pi}I(s)ds>0.
	$$
\end{deff}

\begin{thm}\label{R>1}
Assume that $\lambda>0$. Then for any initial
data
$(u,v,\pi)\in \mathbb{R}_+^{2,\circ}\times\M$,  system \eqref{eq-main} is permanent (in mean).

\end{thm}

\begin{proof}
	We first prove that all invariant measures of $(S(t),I(t),\Pi(t)$ concentrate on $\R^{2,\circ}\times\M$.
	We assume by contradiction that there is no invariant measure on $\R^{2,\circ}_+$ of $(S(t),I(t))$.
Therefore, there is no invariant measure on $\R^{2,*}_+$ since the solutions starting in $\R^{2,*}_+$ will enter and remain in $\R^{2,\circ}_+$ due to Theorem \ref{thm-exi}.
	As a result, $\hat\mu\times\bdelta\times\Phi^*$ (the unique invariant on the boundary $\R_+\times\{0\}\times\M$) is the unique invariant probability measure of the process $\{S(t),I(t),\Pi_t\}$ on $\R^2_+\times\M$.
	Therefore, by applying \cite[Lemma 3.4]{HN18}, we have
	\begin{equation}\label{limity}
	\begin{aligned}
	\lim_{t\to\infty} \frac1t&\Ei\int_0^t\int_0^1 f(x,S(s),I(s))\Pi_s(dx)ds\\
	&=\int_{\R^2_+\times\M} \int_0^1f(x,y,i)\hat\mu(dy)\bdelta(di)\Phi^*(d\nu)\\
	&=\int_0^\infty\int_0^1 f(x,y,0)\mu^*(dx)\hat\mu(dy).
	\end{aligned}
	\end{equation}
	and
	\begin{equation}\label{limitx}
	\begin{aligned}
	\lim_{t\to\infty} \frac1t&\Ei\int_0^t \int_0^1 xh(x,S(s),I(s))\Pi_s(dx)ds\\
	=&\int_{\R^2_+\times\M} \int_0^1xh(x,y,i)\nu(dx)\hat\mu(dy)\bdelta(di)\Phi^*(d\nu)\\
	&=\int_{\R_+\times\M} \int_0^1xh(x,y,0)\nu(dx)\hat\mu(dy)\Phi^*(d\nu)\\
	&=\int_0^\infty\int_0^1 xh(x,y,0)\mu^*(dx)\hat\mu(dy).
	\end{aligned}
	\end{equation}
	On the other hand,
	$$
	\begin{aligned}
	\Ei\frac{\ln Y(t)}t=&\E_{x,y}\frac{\ln v}t-c_2+\frac{1}{t}\E_{x,y}\bigg(\int_0^t \int_0^1f(x,S(s),I(s))\Pi_s(dx)ds\\
	&\qquad+\int_0^t\int_0^1xh(x,S(s),I(s))\Pi_s(dx)ds\bigg)
	+ \E_{x,y} \frac{\sigma_2W_2(t)}t.
	\end{aligned}
	$$
	As a result,
	we have
	$$
	\begin{aligned}
	\lim_{t\to\infty}&\Ei\frac{\ln I(t)}t\\
	=&
	-c_2+\lim_{t\to\infty}\E_{x,y }\frac1t\left(\int_0^t\int_0^1 f(x,S(s),I(s))\Pi_s(dx)ds+\int_0^t\int_0^1xh(x,S(s),I(s))\Pi_s(dx)ds\right)\\
	=&-c_2+\int_0^\infty \int_0^1f(x,y,0)\mu^*(dx)\hat\mu(dy)+\int_0^\infty\int_0^1 xh(x,y,0)\mu^*(dx)\hat\mu(dy)\\
	=&\lambda>0.
	\end{aligned}
	$$
	This contradicts the fact that
	$$\lim_{t\to\infty}\Ei\frac{\ln I(t)}t\leq \lim_{t\to\infty}\Ei\frac{ I(t)}t=0$$
	because $\ln y\leq y$ while Lemma \ref{lem2.1} implies $\lim_{t\to\infty}\Ei\frac{ I(t)}t=0$.
	As a result, all invariant measures of $(S(t),I(t),\Pi(t))$ concentrate on $\R^{2,\circ}_+\times\M_{\Phi^*}$ (the existence of an invariant measure follows from Lemma \ref{lem2.1}).
	
	By the moment boundedness in Lemma \ref{lem2.1}, we obtain that there exists a sequence $T_k\to\infty$ as $k\to\infty$ such that
	$$
	\liminf_{t\to\infty}\frac 1t\int_0^t\E_{u,v,\pi}S(s)ds=\lim_{k\to\infty}\frac1{T_k}\int_0^{T_k}\E_{u,v,\pi}S(s)ds,
	$$
	$$\liminf_{t\to\infty}\frac 1t\int_0^t\E_{u,v,\pi}I(s)ds=\lim_{k\to\infty}\frac1{T_k}\int_0^{T_k}\E_{u,v,\pi}I(s)ds.
	$$
	Therefore, by applying \cite[Lemma 3.4]{HN18} and the fact that all invariant measures of $(S(t),I(t),\Pi(t))$ concentrate on $\R^{2,\circ}_+\times\M_{\Phi^*}$, we obtain the desired result.
\end{proof}

\subsection{Hidden Markov chain}
\label{sec:dis}
This section is devoted to the case when the signal process $\alpha(t)$ takes values in a  finite set $\{m_1,\dots,m_{n^*}\}\subset[0,1]$ and admits a unique invariant measure $(\mu^*_1,\dots,\mu^*_n)\in\mathcal S_{n^*}$.
In this case,  system \eqref{eq-main} is replaced by \eqref{eq-main-dis-1}.
We first have the following well-posedness and other
preliminary results.

\begin{thm} Consider  system \eqref{eq-main-dis-1}.
		For any $(u, v,e)\in\R_+^2\times \mathcal S_{n^*}$, there exists a unique global solution
		to system \eqref{eq-main-dis-1} with  initial data $(u,v,e)$.
		The three-component process $\{(S(t), I(t),\Pi_t: t\geq0\}$
		is a Markov-Feller process.
		Moreover, we have
		$\PP_{u,v,e}\{S(t)>0, t>0\}=1$ and
		$\PP_{u,0,\pi}\{I(t)=0, t>0\}=1$,
		$\PP_{u,v,e}\{I(t)>0, t>0\}=1$ provided $v>0$.
\end{thm}

\begin{proof}
	The proof of this Theorem is as same as that of Theorems \ref{thm-exi} and \ref{thm-markov}, and is thus omitted.
\end{proof}

To proceed,
we classify the persistence and extinction of  system \eqref{eq-main-dis-1} by
the following threshold $\lambda$:
\begin{equation}\label{eq-lambda-dis}
\lambda:=-c_2+\sum_{k=1}^{n^*}\mu^*_k\int_0^\infty f(m_k,y,0)\hat\mu(dy)+\sum_{k=1}^{n^*} m_k\mu^*_k\int_0^\infty h(m_k,y,0)\hat\mu(dy).
\end{equation}


\begin{thm}\label{thm-same}
The following results hold.
	\begin{itemize}
		\item [\rm{(i)}]
If $\lambda<0$ then for any initial point $(u,v,e)\in\mathbb{R}_+^{2,\circ}\times\mathcal S_{n^*}$,
the number of the infected individuals $I(t)$ tends to zero  at an exponential rate, that is, 
$$\PP_{u,v,e}\left\{\limsup\limits_{t\to\infty}\dfrac{\ln I(t)}{t}=\lambda\right\}=1,$$
and the susceptible class $S(t)$ converges weakly to the solution $\varphi(t)$ on the boundary.

\item [\rm{(ii)}] If $\lambda>0$, then for any initial  point $(u,v,e)\in \mathbb{R}_+^{2,\circ}\times\mathcal S_{n^*}$,
system \eqref{eq-main-dis-1} is permanent (in mean).
\item [\rm{(iii)}]
Moreover, the observable system \eqref{eq-main-dis-1} and original
system \eqref{eq-main-1} share the same threshold for the persistence and extinction.
	\end{itemize}
\end{thm}

\begin{proof}
	The proofs of part (i) and (ii) in this Theorem are the same as that of Theorems \ref{R<1}, and \ref{R>1} and, therefore are omitted.
Part (iii) follows immediately from the formulation of $\lambda$ given in \eqref{eq-lambda-dis} and the threshold for system \eqref{eq-main-1} given in \cite{DNY-SPA}.
\end{proof}

%
%

%

\begin{rem}
Note that in Theorem \ref{thm-same}, we have collected several results. These results may be presented in separate theorems. Given that they are all related to the same hidden process
and that we have carried out an extensive analysis in the last section, it seemed reasonable to collect these results in one theorem.

	In Section \ref{sec:num}, we  present numerical examples for Theorem \ref{thm-same}. The reader can see that although the observable system \eqref{eq-main-dis-1} may not approximate well the original at every point in time, it preserves the longtime properties (permanence or extinction).
\end{rem}

\newcommand{\pre}{{\rm\tiny pre}}
\section{Discussion and Interpretation}\label{sec:disc}
When a pandemic arises, there are usually multiple options one can take in order to  control it. If we apply an extreme policy to control the disease transmission (i.e., we try to reduce the infection rate to be very small or almost $0$), we can certainly control the pandemic.
This can be seen
by looking at $\lambda$ defined in the previous section and noting that if $f(S,I)\approx 0$ then $\lambda\approx -c_2<0$.
However, this type of \textit{highly restrictive policy} may hurt the economy.
It is
important
to
balance public health and the economic issues.
In our context, this is equivalent to ensuring that $\lambda<0$, which ensures the pandemic is under control, but not making $\lambda$ too negative, since in the process of doing this (lockdowns, bankruptcy of businesses, unemployment) the economy could suffer significantly.

\begin{deff}
	We say a proposed threshold $\lambda_1$ is overcautious if it is greater than the exact threshold, that is $\lambda_1>\lambda$.
	We say a proposed threshold $\lambda_1$ is incautious if it is less than the exact threshold, that is $\lambda_1<\lambda$ .
\end{deff}

\begin{rem}
We say	a threshold $\lambda_1>\lambda$ (the exact one) is an overcautious proposed threshold 
because if this threshold is implemented,
we tend to
apply a policy to make $\lambda_1<0$. But this may not be necessary because the exact threshold is $\lambda<\lambda_1$.
	Conversely, a threshold $\lambda_1<\lambda$ is an incautious one because reducing $\lambda_1$ to be less 0 may not be enough and the pandemic would
still not be controlled.
\end{rem}

Let us consider the case when $\alpha(t)$ takes the finitely many values $0=m_1< m_2<\dots<m_{n^*}=1$. We note that our analysis is still true for the general case when $\alpha$ takes values in $[0,1]$.
Recall that we assume $\alpha(t)$ has a unique invariant (discrete) measure $(\mu_1^*,\dots,\mu_{n^*}^*)$.

Since
$\alpha(t)$ is not directly available, we have two options.
One option is to use filtering to estimate $\alpha(t)$ and then consider the corresponding system with filtering. It was shown that this method preserves the longtime behavior of the original system.
In particular, from Section \ref{sec:dis}
the exact threshold for this method is
$$
\lambda=-c_2+\sum_{k=1}^{n^*} \mu^*_k\int_0^\infty f(m_k,y,0)\hat\mu(dy)+\sum_{k=1}^{n^*} m_k\mu^*_k\int_0^\infty h(m_k,y,0)\hat\mu(dy).
$$

Another possible option is to estimate some prediction for $\alpha(t)$ and give that value for $\alpha(t)$.
Let $k_0\in\{1,2,\dots,n^*\}$. Assume that we estimate $\alpha(t)=m_{k_0}$ and consider the system
\begin{equation}\label{eq-main-1-alpha=sth}
\begin{cases}
dS(t)=\Big[a_1-b_1S(t)-I(t)f(m_{k_0},S(t),I(t))\\
\hspace{2cm}-\alpha_0I(t)h(m_{k_0},S(t),I(t))\Big]dt + \sigma_1 S(t) dB_1(t),\\[1ex]
dI(t)=\Big[-b_2I(t) + I(t)f(m_{k_0},S(t),I(t))\\
\hspace{2cm}+m_{k_0}I(t)h(m_{k_0},S(t),I(t))\Big]dt + \sigma_2I(t) dB_2(t).\\
\end{cases}
\end{equation}
Using the results from \cite{DNY-SPA}, the threshold for persistence and extinction of \eqref{eq-main-1-alpha=sth} is given by
$$
\lambda_{pre}=-c_2+\int_0^\infty f(m_{k_0},y,0)\hat\mu(dy)+m_{k_0}\int_0^\infty h(m_{k_0},y,0)\hat\mu(dy).
$$
The following
results tell us when $\lambda_{\pre}$ is an incautious or overcautious threshold.


\begin{prop}\label{thm-51}
	\begin{itemize}
		\item [\rm{(i)}]
	If
	$$
	\begin{aligned}
	&\mu_{k_0}^*\int_0^\infty (f(m_{k_0},y,0)+m_{k_0}h(m_{k_0},y,0))\hat\mu(dy)\\
	&\hspace{2cm}<\sum_{k=1}^{n^*}\mu^*_k\int_0^\infty f(m_k,y,0)\hat\mu(dy)+\sum_{k=1}^{n^*} m_k\mu^*_k\int_0^\infty h(m_k,y,0)\hat\mu(dy)
	\end{aligned}
	$$
	then $\lambda_{\pre}<\lambda$. As a result, $\lambda_{\pre}$ is an incautious threshold.
	\item [\rm{(ii)}]	If
	$$
	\begin{aligned}
	&\mu^*_{k_0}\int_0^\infty (f(m_{k_0},y,0)+m_{k_0}h(m_{k_0},y,0))\hat\mu(dy)\\
	&\hspace{2cm}>\frac 1{1-\mu_{n^*}^*}\Big(\sum_{k=1}^{n^*-1}\mu^*_k\int_0^\infty f(m_k,y,0)\hat\mu(dy)+\sum_{k=1}^{n^*-1} m_k\mu^*_k\int_0^\infty h(m_k,y,0)\hat\mu(dy) \Big)
	\end{aligned}
	$$
	then $\lambda_{\pre}>\lambda$. As a result, $\lambda_{\pre}$ is an overcautious threshold.
		\end{itemize}

\end{prop}


	It is natural to assume that $f(x,y,0)$ is an increasing function w.r.t. $x$ because the higher infection rate in the group of potentially infected individuals would make the disease spread faster.
	Note that this intuition is natural but not
always true because we are examining
functions at the boundary, i.e., there is no infected group.

\begin{asp}\label{asp-natural}
	For each fixed $y$, the functions $f(x,y,0)$ and $h(x,y,0)$ are increasing in $x$.
\end{asp}

Under this natural assumption, we will see that assuming that all potentially infected individuals are infected will lead to an overcautious policy. Conversely, it will be incautious if we assume that all individuals with hidden infection status are free of the disease.
We will make this analysis clearer in the following two subsections.

\subsection{The overcautious case: assuming that all individuals in the hidden group are infected}
Suppose we do not use the filtering to consider the observable problem, make the assumption $\alpha(t)=1, t\geq 0$ and consider the system \eqref{eq-main-1-alpha=sth} with $m_{k_0}=m_{n^*}=1$.
Under Assumption \ref{asp-natural}, this is an overcautious prediction.
The following theorem is consistent with this intuition.
%

\begin{prop}\label{thm-53}	
		Under Assumption \ref{asp-natural},
	\begin{equation}\label{asp-lambda-1}
	\begin{aligned}
	&\int_0^\infty (f(1,y,0)+h(1,y,0))\hat\mu(dy)\\
	&\hspace{1cm}>\frac 1{1-\mu_{n^*}^*}\Big(\sum_{k=1}^{n^*-1}\mu^*_k\int_0^\infty f(m_k,y,0)\hat\mu(dy)+\sum_{k=1}^{n^*-1} m_k\mu^*_k\int_0^\infty h(m_k,y,0)\hat\mu(dy)\Big)
	\end{aligned}
	\end{equation}
As a result, $\lambda_{\pre}>\lambda$ and $\lambda_{\pre}$ is an overcautious threshold.
\end{prop}

\begin{proof}
	The estimate \eqref{asp-lambda-1} follows immediately from Assumption \ref{asp-natural}. We omit the details here.
\end{proof}

\subsection{The incautious case: assuming that all individuals in the hidden group are not infected}
If we make the assumption that $\alpha(t)=0, t\geq 0$ and consider the system \eqref{eq-main-1-alpha=sth} with $m_{k_0}=m_1=0$, this is an incautious prediction (under Assumption \ref{asp-natural}).
The following result is consistent with this fact.
%

\begin{prop}\label{thm-54}
	Under Assumption \ref{asp-natural},
	$$
	\begin{aligned}
	&\int_0^\infty (f(0,y,0)+h(0,y,0))\hat\mu(dy)\\
	&\hspace{2cm}<\frac 1{1-\mu_1^*}\Big(\sum_{k=2}^{n^*}\mu^*_k\int_0^\infty f(m_k,y,0)\hat\mu(dy)+\sum_{k=2}^{n^*} m_k\mu^*_k\int_0^\infty h(m_k,y,0)\hat\mu(dy)\Big)
	\end{aligned}
	$$
As a result, $\lambda_{\pre}<\lambda$ and $\lambda_{\pre}$ is an incautious threshold.
\end{prop}


\section{Examples and Simulations}\label{sec:exp}
\subsection{A Simple Example}
In this section, we consider a simple example.
Assume that all the individuals in the hidden class have the same status (susceptible or infected).
In other words, the signal Markov process $\alpha(t)$ has only two possible states, $0$ or $1$, (corresponding to the case that all individuals in hidden class are disease-free or infected, respectively).
Assume that $\alpha(t)$ has the generator
$$
Q=
\begin{bmatrix}
-q_1 &q_1\\
q_2&-q_2
\end{bmatrix}.
$$
We can only observe $\alpha(t)$ via the observation process $y(t)$ given by
$$
dy(t)=g(\alpha(t))dt+dW(t),\quad y(0)=0,
$$
where $g:\{0,1\}\to\R$ and set $g_1=g(0)$, $g_2=g(1)$. Let $e(t)=\E[\1_{\{\alpha(t)=0\}}|\F_t^y]$.

The dynamics of this epidemic system under the (Wonham) filter is described by
\begin{equation}\label{eq-ex-1}
\begin{cases}
dS(t)=\Big[a_1-b_1S(t)-I(t)f(0,S(t),I(t))e(t)-I(t)f(1,S(t),I(t))(1-e(t))\\
\hspace{2cm}-I(t)h(1,S(t),I(t))(1-e(t))\Big]dt + \sigma_1 S(t) dB_1(t),\\[1ex]
dI(t)=\Big[-b_2I(t) + I(t)f(0,S(t),I(t))e(t)+I(t)f(1,S(t),I(t))(1-e(t))\\
\hspace{2cm}+I(t)h(1,S(t),I(t))(1-e(t))\Big]dt + \sigma_2I(t) dB_2(t),\\
de(t)=[q_2-(q_1+q_2)e(t)]dt+(g_1-g_2)e(t)(1-e(t))d\bar W(t).
\end{cases}
\end{equation}
We will also assume that $q_1,q_2>0$ and $g:=g_1-g_2\neq 0$ since the other cases are trivial.
Consider the equation of the third component
\begin{equation}\label{eq-e}
de(t)=[q_2-(q_1+q_2)e(t)]dt+ge(t)(1-e(t))d\bar W(t),\quad e(0)\in [0,1].
\end{equation}
By using the Lyapunov functional method as in Theorem \ref{thm-exi}, we obtain the following result.

\begin{thm}
	For any initial value in $[0,1]$, the equation \eqref{eq-e} has a unique solution and
	$$\PP\{e(t)\in (0,1), \forall t>0\}=1.$$
\end{thm}
Now, by solving the Fokker-Plank equation, equation \eqref{eq-e} has a unique invariant measure supported in $[0,1]$ with density given by
$$
\phi^*(x)=Ce^{-\frac{d_1}{1-x}}(1-x)^{2(d_1-d_2-1)} e^{-\frac{d_2}{x}}x^{2(d_2-d_1-1)},\;x\in(0,1),
$$
where $d_1=\frac{q_1}{g^2}$,
$d_2=\frac{q_2}{g^2}$, and
$C$ is a normalizing constant.

As developed in the main results, the threshold $\lambda$ is defined by
\begin{equation}\label{eq-ex-lambda}
\begin{aligned}
\lambda
=&-c_2+ \frac{q_2}{q_1+q_2}\int_0^{\infty}f(0,y,0)\hat\mu(dy)\\
&+\frac{q_1}{q_1+q_2}\int_0^{\infty}f(1,y,0)\hat\mu(dy)+\frac{q_1}{q_1+q_2}\int_0^\infty h(1,y,0)\hat\mu(dy),
\end{aligned}
\end{equation}
where
$\hat\mu$ is the invariant measure with the density given by \eqref{hpp}.

\begin{thm} Consider system \eqref{eq-ex-1} and $\lambda$ as in \eqref{eq-ex-lambda}.
	\begin{itemize}
		\item
		If $\lambda<0$ then for any initial point $(u,v,e_0)\in\mathbb{R}_+^{2,\circ}\times\mathcal S_{2}$,
		the number of the infected individuals $I(t)$ tends to zero at an exponential rate, i.e., 
		$$\PP_{u,v,e_0}\left\{\limsup\limits_{t\to\infty}\dfrac{\ln I(t)}{t}\leq \lambda\right\}=1,$$
		and the susceptible class $S(t)$ converges weakly to the solution $\varphi(t)$ on the boundary.
		\item If $\lambda>0$, then for any initial  point $(u,v,e_0)\in \mathbb{R}_+^{2,\circ}\times\mathcal S_{2}$, all invariant probability measures of
		the solution $(S(t),I(t),e(t))$
		 concentrate on $\mathbb{R}_+^{2,\circ}\times[0,1]$.
		Moreover, the system \eqref{eq-ex-1} is permanent.
	\end{itemize}
\end{thm}

\subsection{Numerical Examples}\label{sec:num}
In this section, we consider a simple example and provide some numerical simulations.
Consider the equation \eqref{eq-ex-1} with $f(x,s,i)=m_1(x)s$,
 $h(x,s,i)=\frac{m_2(x)s}{1+s+i}$, $g(x)=x$, ($x\in\{0,1\}$), i.e., consider
\begin{equation}\label{eq-ex-2}
\begin{cases}
dS(t)=\Big[a_1-b_1S(t)-m_1(0)S(t)I(t)e(t)-m_1(1)S(t)I(t)(1-e(t))
\\
\hspace{2cm}-\dfrac{m_2(1)(1-e(t))S(t)I(t)}{1+S(t)+I(t)}\Big]dt + \sigma_1 S(t) dB_1(t),\\[2ex]
dI(t)=\Big[-b_2I(t) + m_1(0)S(t)I(t)e(t)+m_1(1)S(t)I(t)(1-e(t))\\
\hspace{2cm}+\dfrac{m_2(1)(1-e(t))S(t)I(t)}{1+S(t)+I(t)}\Big]dt + \sigma_2I(t) dB_2(t),\\
de(t)=[q_2-(q_1+q_2)e(t)]dt+e(t)(1-e(t))d\bar W(t).
\end{cases}
\end{equation}
The above is the corresponding system
with a filtering of hidden Markov chain $\alpha(t)$, whereas
the following
system is one under hidden process $\alpha(t)$
\begin{equation}\label{eq-ex-3}
\begin{cases}
dS(t)=\Big[a_1-b_1S(t)-m_1(\alpha(t))S(t)I(t)\\
\hspace{2cm}-\dfrac{\alpha(t)m_2(\alpha(t))S(t)I(t)}{1+S(t)+I(t)}\Big]dt + \sigma_1 S(t) dB_1(t),\\[2ex]
dI(t)=\Big[-b_2I(t) + m_1(\alpha(t))S(t)I(t)\\
\hspace{2cm}+\dfrac{\alpha(t)m_2(\alpha(t))S(t)I(t)}{1+S(t)+I(t)}\Big]dt + \sigma_2I(t) dB_2(t).\\
\end{cases}
\end{equation}
Here $\alpha(t)$ is the Markov chain taking values in $\{0,1\}$ with the generator $Q$.

Suppose one does not use the filtering to consider the corresponding observable system \eqref{eq-ex-2}, and considers instead the system under some
predictions for $\alpha(t)$.
If one uses the incautious prediction and assumes that $\alpha(t)=0$ (i.e., considers all individuals in the hidden group not to be infected), the corresponding system is
\begin{equation}\label{eq-ex-4}
\begin{cases}
dS(t)=\Big[a_1-b_1S(t)-m_1(0)S(t)I(t)\Big]dt + \sigma_1 S(t) dB_1(t),\\[2ex]
dI(t)=\Big[-b_2I(t) + m_1(0)S(t)I(t)\Big]dt + \sigma_2I(t) dB_2(t).\\
\end{cases}
\end{equation}

If one uses the overcautious prediction and assumes that $\alpha(t)=1$ (i.e., considers all individuals in the hidden group to be infected), the corresponding system is
 \begin{equation}\label{eq-ex-5}
 \begin{cases}
 dS(t)=\Big[a_1-b_1S(t)-m_1(1)S(t)I(t)-\dfrac{m_2(1)S(t)I(t)}{1+S(t)+I(t)}\Big]dt + \sigma_1 S(t) dB_1(t),\\[2ex]
 dI(t)=\Big[-b_2I(t) + m_1(1)S(t)I(t)+\dfrac{m_2(1)S(t)I(t)}{1+S(t)+I(t)}\Big]dt + \sigma_2I(t) dB_2(t).\\
 \end{cases}
 \end{equation}

\begin{example}
	Consider \eqref{eq-ex-2} with
	$a_1=0.5,
	b_1=1,
	\sigma_1=1,
	b_2=2,
	\sigma_2=0.5,
	m_1(0)=0.1,
	m_1(1)=4,
	m_2=0.1,
	q_1=5,
	q_2=25.
	$
Our computation shows that
$\lambda=-1.7252$, the (exact) threshold determining the longtime behavior (persistence and extinction) for both systems \eqref{eq-ex-2} and \eqref{eq-ex-3}.
Similarly, we can compute
$\lambda_0=-2.0750$, the threshold for system \eqref{eq-ex-4} and
$\lambda_1=0.0241$, the threshold for system \eqref{eq-ex-5}.

Note that $\lambda_1>\lambda>\lambda_0$ and as a result $\lambda_1$ is an overcautious threshold. In this case, if we use system \eqref{eq-ex-5}, we may conclude that the disease may not be controlled although it will indeed be controlled. Some unnecessarily restrictive policy might be chosen and this might lead to economic downturns.
Conversely, $\lambda_0$ is an incautious threshold which would lead to overly optimistic expectations.

	As our theoretical results show the number of infected $I(t)$ will tend to $0$ as $t\to\infty$, i.e., the infected group goes extinct.
	We have also shown that  systems \eqref{eq-ex-2} and \eqref{eq-ex-3} have the same longtime behavior.
	We provide the numerical simulations for this example in Figure \ref{fig-1}.
	
	\begin{figure}[h!tb]
		\centering     
		\subfigure{\includegraphics[width=.5\textwidth]{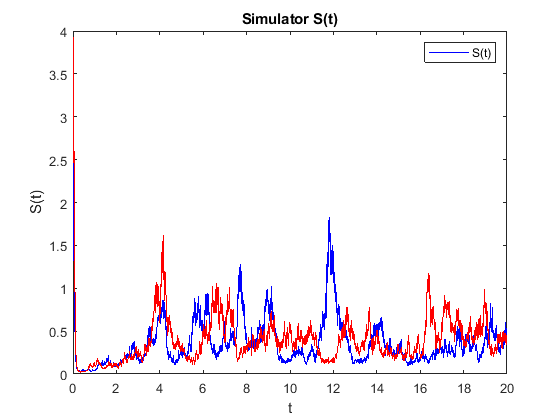}}\hspace{-1.5em}
		\subfigure{\includegraphics[width=.5\textwidth]{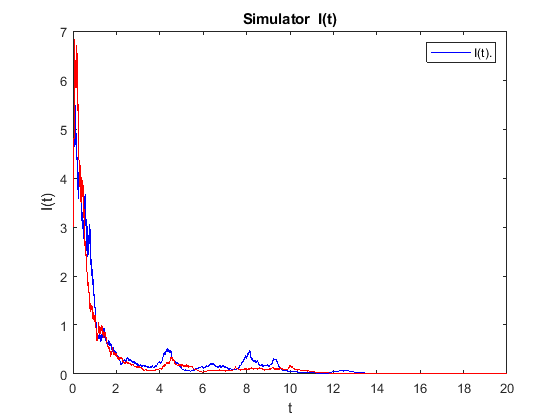}}
		\subfigure{\includegraphics[width=.5\textwidth]{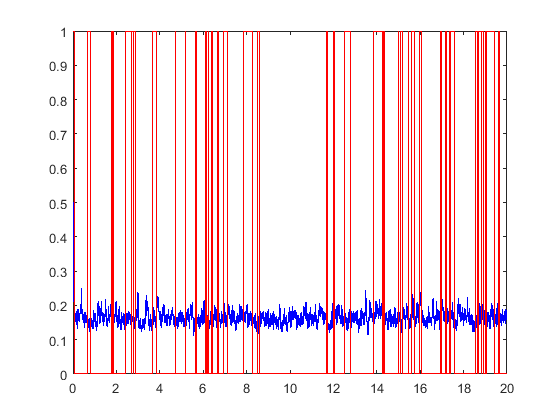}}
		\caption{Left top subfigure:
Sample paths of $S(t)$, in which the red curve is for
system under the hidden $\alpha(t)$ in \eqref{eq-ex-3}, the blue curve corresponding to the known $\alpha(t)$ in
\eqref{eq-ex-2}. Right top subfigure: Sample paths of $I(t)$, the red curve is
for
system \eqref{eq-ex-3} with hidden state, the blue curve is
that of \eqref{eq-ex-2} with known $\alpha(t)$. Bottom
subfigure:
the sample path of the Markov chains, the red is of the signal process $\alpha(t)$, the blue is its filter $e(t)$.}\label{fig-1}
	\end{figure}
\end{example}

\begin{example}
	Consider \eqref{eq-ex-2} with
	$a_1=10,
	b_1=1,
	\sigma_1=1,
	b_2=3,
	\sigma_2=1,
	m_1(0)=0.1,
	m_1(1)=2,
	m_2(1)=0.1,
	q_1=10,
	q_2=1.
	$
Direct computations yield
	$\lambda=14.8522$, the (exact) threshold determining the longtime behavior (persistence and extinction) for both systems \eqref{eq-ex-2} and \eqref{eq-ex-3}.
	Similarly, we can compute
	$\lambda_{0}=-2.5000$, the threshold for system \eqref{eq-ex-4}
	and $\lambda_1=16.5874$, the threshold for system \eqref{eq-ex-5}.
Because $\lambda_1>\lambda>\lambda_0$, we note that $\lambda_1$ is an overcautious threshold. Conversely, $\lambda_0$ is an incautious threshold. [It is readily seen that  the system is actually permanent, i.e., the disease will not be controlled but $\lambda_0$ recommends the disease will be extinct.
Conversely, $\lambda_1$ is an overcautious threshold which would lead to overly pessimistic expectations.]
	
	Applying our theoretical results to this example, we get that $I(t)$ never converges to $0$ and the system is permanent.
		As before, the systems \eqref{eq-ex-2} and \eqref{eq-ex-3} have the same longtime behavior.
	The numerical simulations of this example are provided in Figures \ref{fig-2} and \ref{fig-3}.
	\begin{figure}[h!tb]
		\centering     
		\subfigure{\includegraphics[width=.5\textwidth]{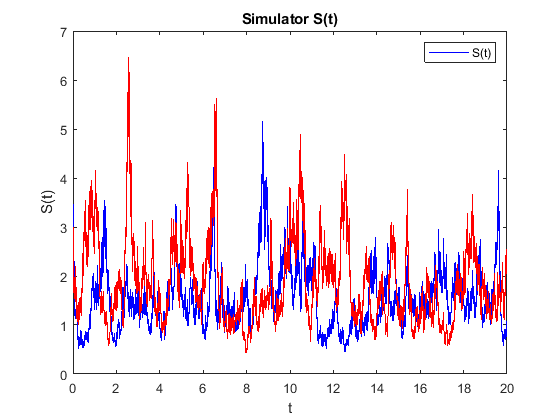}}\hspace{-1.5em}
		\subfigure{\includegraphics[width=.5\textwidth]{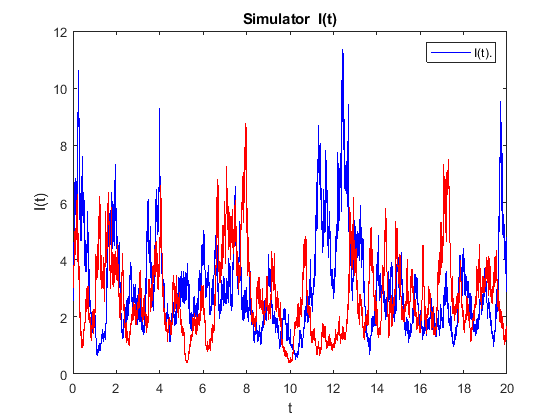}}
		\subfigure{\includegraphics[width=.5\textwidth]{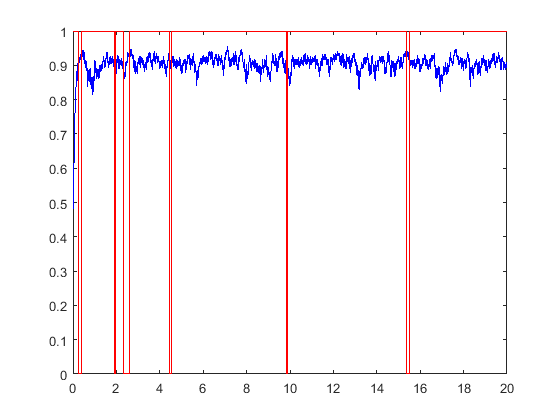}}
		\caption{Left top subfigure:
Sample path of $S(t)$, the red curve is for the
system under noisy observation of $\alpha(t)$, the blue is
the
system with precise known value of $\alpha(t)$. Right top subfigure:
Sample path of $S(t)$, the red curve is for
system under noisy observation, the blue curve is
for system with precise known value of $\alpha(t)$. Bottom subfigure:
the sample path of the Markov chain, the red curve is of the signal process, the blue curve is its filter.}\label{fig-2}
	\end{figure}

\begin{figure}[h!tb]
	\centering     
	\subfigure{\includegraphics[width=.5\textwidth]{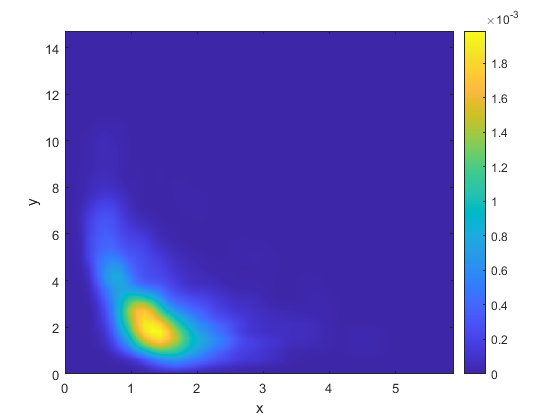}}\hspace{-1.5em}
	\subfigure{\includegraphics[width=.5\textwidth]{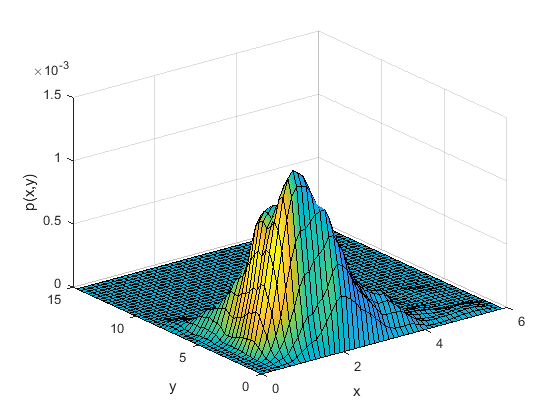}}
	\subfigure{\includegraphics[width=.5\textwidth]{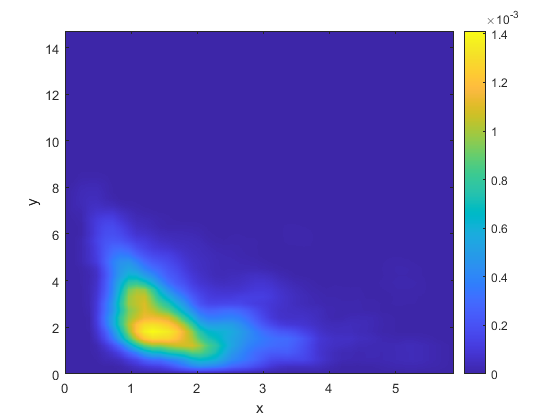}}\hspace{-1.5em}
	\subfigure{\includegraphics[width=.5\textwidth]{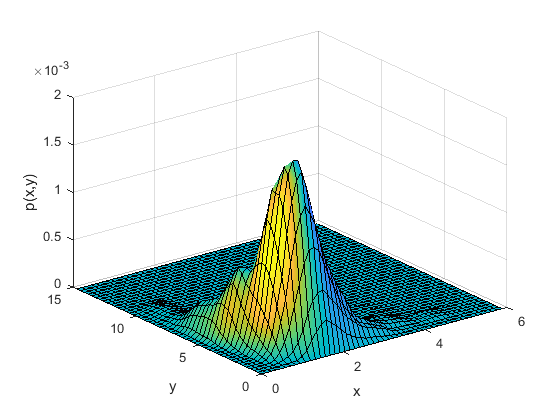}}
	\caption{Left top and right top subfigures:
The density of the invariant probability measure (the marginal one in the space of $(S(t),I(t))$) of \eqref{eq-ex-3} in 2D and 3D settings, respectively.
Bottom subfigure :
	The density of the invariant probability measure (the marginal one in the space of $(S(t),I(t))$) of \eqref{eq-ex-4} in 2D and 3D settings, respectively.}\label{fig-3}
\end{figure}

\end{example}


\FloatBarrier

\end{document}